\pdfoutput=1
\documentclass[a4paper, reqno]{amsart}
\usepackage[utf8]{inputenc}
\usepackage{graphicx}

\usepackage[UKenglish]{babel}

\DeclareSymbolFont{symbolsC}{U}{pxsyc}{m}{n}
\DeclareMathSymbol{\coloneqq}{\mathrel}{symbolsC}{"42}

\usepackage{comment}
 
\usepackage{csquotes}

\usepackage[symbols,nogroupskip,automake]{glossaries-extra}

\usepackage[
backend=bibtex,
style=alphabetic,
sorting=nyt
]{biblatex}
\title{The Moduli Space of marked supersingular Enriques surfaces}
\author{Kai Behrens \\ Zentrum Mathematik M11 \\ Technische Universit\"at M\"unchen \\ behrensk@ma.tum.de}

\usepackage[usenames,dvipsnames]{color}

\usepackage{amsthm,amsfonts,amssymb,amsmath}
\usepackage[all]{xy}
\usepackage{xr-hyper}
\usepackage[colorlinks= false,
   citecolor=Red,
   linkcolor=Red,
   urlcolor=Blue]{hyperref}
\usepackage{verbatim}
\usepackage{pdfsync}
\usepackage{pigpen}
\usepackage{stmaryrd}
\usepackage{mathtools}

\newcommand{\CA}{{\mathcal {A}}}

\newcommand{\CE}{{\mathcal {E}}}
\newcommand{\CF}{{\mathcal {F}}}

\newcommand{\CL}{{\mathcal {L}}}
\newcommand{\CM}{{\mathcal {M}}}

\newcommand{\CO}{{\mathcal {O}}}
\newcommand{\CP}{{\mathcal {P}}}
\newcommand{\CQ}{{\mathcal {Q}}}
\newcommand{\CR}{{\mathcal {R}}}
\newcommand{\CS}{{\mathcal {S}}}

\newcommand{\CX}{{\mathcal {X}}}
\newcommand{\CY}{{\mathcal {Y}}}

\newcommand{\Aut}{{\mathrm{Aut}}}

\newcommand{\Pic}{\mathrm{Pic}}

\newcommand{\Spec}{{\mathrm{Spec}}\,}

\newcommand{\lra}{\longrightarrow}

\newcommand{\po}{\ar@{}[dr]|{\text{\pigpenfont R}}}
\newcommand{\pb}{\ar@{}[dr]|{\text{\pigpenfont J}}}

\newcommand{\claim}{\par\noindent{\sc Claim.}\quad}
\newcommand{\prfclaim}{\par\noindent{\sc Proof of claim.}\quad}

\newtheorem{theorem}{Theorem}[section]
\newtheorem{proposition}[theorem]{Proposition}
\newtheorem{lemma}[theorem]{Lemma}

\newtheorem{corollary}[theorem]{Corollary}

\newtheorem*{theorem*}{Theorem}
\newtheorem*{proposition*}{Proposition}

\theoremstyle{definition}
\newtheorem{definition}[theorem]{Definition}
\newtheorem*{definition*}{Definition}

\newtheorem{remark}[theorem]{Remark}
\newtheorem*{remark*}{Remark}

\numberwithin{equation}{section}

\usepackage{tikz}
\usepackage{verbatim}
\usetikzlibrary{arrows,chains,matrix,positioning,scopes}
\makeatletter
\tikzset{join/.code=\tikzset{after node path={%
\ifx\tikzchainprevious\pgfutil@empty\else(\tikzchainprevious)%
edge[every join]#1(\tikzchaincurrent)\fi}}}
\makeatother
\tikzset{>=stealth',every on chain/.append style={join},
         every join/.style={->}}
\tikzstyle{labeled}=[execute at begin node=$\scriptstyle,
   execute at end node=$]

\makeatletter
\let\save@mathaccent\mathaccent
\newcommand*\if@single[3]{%
  \setbox0\hbox{${\mathaccent"0362{#1}}^H$}%
  \setbox2\hbox{${\mathaccent"0362{\kern0pt#1}}^H$}%
  \ifdim\ht0=\ht2 #3\else #2\fi
  }
\newcommand*\rel@kern[1]{\kern#1\dimexpr\macc@kerna}
\newcommand*\widebar[1]{\@ifnextchar^{{\wide@bar{#1}{0}}}{\wide@bar{#1}{1}}}
\newcommand*\wide@bar[2]{\if@single{#1}{\wide@bar@{#1}{#2}{1}}{\wide@bar@{#1}{#2}{2}}}
\newcommand*\wide@bar@[3]{%
  \begingroup
  \def\mathaccent##1##2{%
    \let\mathaccent\save@mathaccent
    \if#32 \let\macc@nucleus\first@char \fi
    \setbox\z@\hbox{$\macc@style{\macc@nucleus}_{}$}%
    \setbox\tw@\hbox{$\macc@style{\macc@nucleus}{}_{}$}%
    \dimen@\wd\tw@
    \advance\dimen@-\wd\z@
    \divide\dimen@ 3
    \@tempdima\wd\tw@
    \advance\@tempdima-\scriptspace
    \divide\@tempdima 10
    \advance\dimen@-\@tempdima
    \ifdim\dimen@>\z@ \dimen@0pt\fi
    \rel@kern{0.6}\kern-\dimen@
    \if#31
      \overline{\rel@kern{-0.6}\kern\dimen@\macc@nucleus\rel@kern{0.4}\kern\dimen@}%
      \advance\dimen@0.4\dimexpr\macc@kerna
      \let\final@kern#2%
      \ifdim\dimen@<\z@ \let\final@kern1\fi
      \if\final@kern1 \kern-\dimen@\fi
    \else
      \overline{\rel@kern{-0.6}\kern\dimen@#1}%
    \fi
  }%
  \macc@depth\@ne
  \let\math@bgroup\@empty \let\math@egroup\macc@set@skewchar
  \mathsurround\z@ \frozen@everymath{\mathgroup\macc@group\relax}%
  \macc@set@skewchar\relax
  \let\mathaccentV\macc@nested@a
  \if#31
    \macc@nested@a\relax111{#1}%
  \else
    \def\gobble@till@marker##1\endmarker{}%
    \futurelet\first@char\gobble@till@marker#1\endmarker
    \ifcat\noexpand\first@char A\else
      \def\first@char{}%
    \fi
    \macc@nested@a\relax111{\first@char}%
  \fi
  \endgroup
}
\makeatother

\makeglossaries
\glsxtrnewsymbol[description={lattice $U_2 \oplus E_8(-1)$}]{Gamma}{\ensuremath{\Gamma}}

\glsxtrnewsymbol[description={Neron Severi lattice of Enriques surface, $\Gamma \oplus \mathbb{Z}/2\mathbb{Z}$}]{Gamma'}{\ensuremath{\Gamma'}}

\glsxtrnewsymbol[description={K3 lattice of Artin invariant $\sigma$}]{N_s}{\ensuremath{N_{\sigma}}}

\glsxtrnewsymbol[description={category of algebraic spaces over $\mathbb{F}_p$}]{A_p}{\ensuremath{\CA_{\mathbb{F}_p}}}

\glsxtrnewsymbol[description={families of $N_{\sigma}$-marked ssg.\ K3 surfaces}]{S_s}{\ensuremath{\CS_{\sigma}}} 

\glsxtrnewsymbol[description={period space of K3 crystals with Artin invariant $\leq \sigma$; characteristic generatrices of $N_{\sigma}$}]{M_s}{\ensuremath{\CM_{\sigma}}} 

\glsxtrnewsymbol[description={set of embeddings $N_{\sigma} \hookrightarrow N_{\sigma'}$}]{R_ss}{\ensuremath{R_{\sigma', \sigma}}}

\glsxtrnewsymbol[description={characteristic generatrices of Artin invariant $\sigma'$ in $N_{\sigma}$}]{M_ss}{\ensuremath{\CM^{\sigma'}_{\sigma}}} 

\glsxtrnewsymbol[description={scheme lying under $\CS'_{\gamma}$}]{M'_g}{\ensuremath{\CM'_{\gamma}}} 

\glsxtrnewsymbol[description={set of embeddings $\gamma \colon \Gamma(2) \hookrightarrow
N_{\sigma}$}]{R_s}{\ensuremath{R_{\sigma}}}

\glsxtrnewsymbol[description={families in $\CS_{\sigma}$ of Artin invariant
$\sigma'$}]{S_ss}{\ensuremath{\CS^{\sigma'}_{\sigma}}} 

\glsxtrnewsymbol[description={families in $\CS_{\sigma}$ s.th.\ induced $\Gamma(2)$-marking has an ample bundle}]{S'_g}{\ensuremath{\CS'_{\gamma}}} 

\glsxtrnewsymbol[description={stack of families of ssg.\ K3 surfaces with a sublattice $\CR$ in $\Pic$ s.th.\ induced $\Gamma(2)$-marking has an ample bundle and no $-2$-vector in complement}]{St''_g}{\ensuremath{\tilde{\mathfrak{S}}''_{\gamma}}} 

\glsxtrnewsymbol[description={stack of equivalence classes of characteristic subspaces lying under $\tilde{\mathfrak{S}}''_{\gamma}$}]{Mft''_g}{\ensuremath{\tilde{\mathfrak{M}}''_{\gamma}}} 

\glsxtrnewsymbol[description={scheme lying under $\tilde{\CS}'_{\gamma}$}]{Mt'_g}{\ensuremath{\widetilde{\CM}'_{\gamma}}}

\glsxtrnewsymbol[description={families in $\CS'_{\gamma}$ s.th.\ induced $\Gamma(2)$-marking has no $-2$-vector in complement}]{S''_g}{\ensuremath{\CS''_{\gamma}}} 

\glsxtrnewsymbol[description={scheme lying under $\tilde{\CS}''_{\gamma}$}]{Mt''_g}{\ensuremath{\widetilde{\CM}''_{\gamma}}}

\glsxtrnewsymbol[description={stack of families of ssg.\ K3 surfaces admitting an $N_{\sigma}$-marking, with a $\Gamma(2)$-marking that has an ample bundle and no $-2$-vector in complement}]{Et_s}{\ensuremath{\tilde{\mathfrak{E}}_{\sigma}}} 

\glsxtrnewsymbol[description={scheme lying under $\tilde{\CE}_{\sigma} \cong \CE_{\sigma}$; coarse moduli space of supersingular Enriques surfaces}]{Q_s}{\ensuremath{\CQ_{\sigma}}}

\glsxtrnewsymbol[description={families of $\Gamma'$-marked ssg.\ Enriques surfaces s.th.\ K3 cover admits $N_{\sigma}$-marking; isomorphic to $\tilde{\CE}_{\sigma}$}]{E_s}{\ensuremath{\CE_{\sigma}}} 

\glsxtrnewsymbol[description={compactification of $\tilde{\CE}_{\sigma}$}]{Etd_s}{\ensuremath{\tilde{\CE}^{\dag}_{\sigma}}} 

\glsxtrnewsymbol[description={compactification of $\CQ_{\sigma}$}]{Qtd_s}{\ensuremath{\CQ^{\dag}_{\sigma}}} 

\glsxtrnewsymbol[description={families in $\tilde{\CE}_{\sigma}$ s.th.\ $\Gamma(2)$-marking factorizes over an $N_{\sigma}$-marking}]{Et'_s}{\ensuremath{\tilde{\CE}'_{\sigma}}} 

\glsxtrnewsymbol[description={families in $\CE_{\sigma}$ s.th.\ induced $\Gamma(2)$-marking of covering K3 factorizes over an $N_{\sigma}$-marking}]{E'_s}{\ensuremath{\CE'_{\sigma}}} 

\glsxtrnewsymbol[description={scheme lying under $\tilde{\CE}'_{\sigma}$}]{Q'_s}{\ensuremath{\CQ'_{\sigma}}}

\glsxtrnewsymbol[description={coarse moduli space of the stack $\tilde{\mathfrak{S}}''_{\gamma}$}]{Stc''_s}{\ensuremath{\tilde{\CS}''_{\sigma}}} 

\glsxtrnewsymbol[description={coarse moduli space of the stack $\tilde{\mathfrak{E}}_{\sigma}$; a postiori isomorphic to $\tilde{\mathfrak{E}}_{\sigma}$}]{Etc''_s}{\ensuremath{\tilde{\CE}_{\sigma}}} 

\begin{document}
 \address{Kai Behrens \\ Technische Universit\"at M\"unchen \\ Zentrum Mathematik M11 \\ Boltzmannstra\ss e 3 \\ 85748 Garching \\ Germany}
 \email{behrensk@ma.tum.de}
\begin{abstract} We construct a moduli space of adequately marked Enriques surfaces that have a supersingular K3 cover over fields of characteristic $p \geq 3$. We show that this moduli space exists as a scheme locally of finite type over $\mathbb{F}_p$. Moreover, there exists a period map from this moduli space to a period scheme and we obtain a Torelli theorem for supersingular Enriques surfaces.
\end{abstract}
\maketitle
\section*{Introduction}
Over the complex numbers there exists a Torelli theorem for K3 surfaces in terms of Hodge cohomology. Moreover, for K3 surfaces together with a polarization of fixed even degree there is a coarse moduli space which is a quasi-projective variety of dimension $19$ over $\mathbb{C}$ \cite{MR0284440}, \cite{MR0447635}.

Over a field of characteristic $p \neq 2$, Enriques surfaces are precisely the quotients of K3 surfaces by fixed point free involutions. Using this connection between Enriques surfaces and K3 surfaces, Namikawa proved a Torelli theorem for complex Enriques surfaces and showed that there is a $10$-dimensional quasi-projective variety which is a coarse moduli space for complex Enriques surfaces \cite{MR771979}. If $Y$ is an Enriques surface, then its Neron-Severi group $\mathrm{NS}(Y)$ is isomorphic to the lattice $\Gamma' = \Gamma \oplus \mathbb{Z}/2\mathbb{Z}$ with $\Gamma= U_2 \oplus E_8(-1)$. By the Torelli theorem for complex K3 surfaces, fixed point free involutions of a K3 surface $X$ can then be characterized in terms of certain embeddings $\Gamma(2) \hookrightarrow \mathrm{NS}(X)$.

Now we turn to characteristic $p >2$. For Enriques surfaces that are quotients of ordinary K3 surfaces over perfect fields of positive characteristic, that is K3 surfaces $X$ with $h(X)=1$, Laface and Tirabassi recently proved a Torelli theorem \cite{laface2019ordinary}. 

For supersingular K3 surfaces over an algebraically closed field of characteristic at least $3$, crystalline cohomology plays a role similar to the role of Hodge cohomology in characteristic zero. Ogus proved a Torelli theorem for supersingular K3 surfaces \cite{MR717616} which shows that supersingular K3 surfaces are determined by their corresponding K3 crystals. For a K3 lattice $N$, an \emph{$N$-marking} of a supersingular K3 surface $X$ is an embedding of lattices $\gamma \colon N \hookrightarrow \mathrm{NS}(X)$. Supersingular K3 surfaces are stratified by the Artin invariant $\sigma$, where $-p^{2\sigma}$ is the discriminant of $\mathrm{NS}(X)$. We always have $1 \leq \sigma \leq 10$ \cite{MR0371899}.

A version of Ogus' Torelli theorem states that for families of $N$-marked supersingular K3 surfaces of Artin invariant at most $\sigma$ there exists a fine moduli space $\CS_{\sigma}$ which is a smooth scheme of dimension $\sigma-1$, locally of finite type, but not separated. There is an étale surjective period map $\pi_{\sigma} \colon \CS_{\sigma} \lra \CM_{\sigma}$ from $\CS_{\sigma}$ to a period scheme $\CM_{\sigma}$. The latter is smooth and projective of dimension $\sigma-1$ and is a moduli space for marked K3 crystals. The functors represented by $\CS_{\sigma}$ and $\CM_{\sigma}$ have interpretations in terms of so-called \emph{characteristic subspaces} of $pN^{\vee}/pN$. If $X$ is a supersingular K3 surface over an algebraically closed field of characteristic $p \geq 3$ and $\iota \colon X \rightarrow X$ is a fixed point free involution, we write $G=\langle \iota \rangle$ for the cyclic group of order $2$ which is generated by $\iota$.
\begin{definition*} A \textit{supersingular Enriques surface} is a surface $Y$ that is the quotient of a supersingular K3 surface $X$ by the action of a fixed point free involution $\iota$. The \emph{Artin invariant} of a supersingular Enriques surface $Y$ is the Artin invariant of the supersingular K3 surface $X$ that universally covers $Y$. \end{definition*}
\begin{remark*} We should remark that our definition of supersingular Enriques surfaces is \emph{not} connected to the classical terminology over fields of characteristic $2$ where a supersingular Enriques surface is an Enriques surface $Y$ with $h^1(\CO_Y) = 1$. Since we always assume the characteristic of the ground field to be odd, no confusion can arise.
\end{remark*}
In this article we construct a fine moduli space for marked supersingular Enriques surfaces. More precisely, writing $\CA_{\mathbb{F}_p}$ for the category of algebraic spaces over $\mathbb{F}_p$, $N_{\sigma}$ for a fixed K3 lattice of Artin invariant $\sigma$ and $\Gamma' = \Gamma \oplus \mathbb{Z}/2\mathbb{Z}$ as above, we study the functor
\begin{align*} \underline{\CE}_{\sigma} \colon \CA_{\mathbb{F}_p}^{\mathrm{op}} &\lra \left(\text{Sets} \right) \\
S &\longmapsto \left\{\begin{array}{l}\text{Isomorphism classes of tuples $(\tilde{f}, \tilde{\gamma}, \CL, \mu)$, where} \\ \text{$\tilde{f} \colon \CY \rightarrow S$ is a family of supersingular Enriques surfaces} \\  \text{with $\Gamma'$-marking $\tilde{\gamma} \colon \Gamma' \rightarrow \Pic_{\CY/S}$ and canonizing datum $(\CL, \mu)$} \\
\text{such that the canonical K3 cover $\CX \rightarrow \CY$} \\
\text{fppf locally admits an $N_{\sigma}$-marking} \end{array} \right\}.\end{align*}
Using the supersingular Torelli theorem, we attack this moduli problem by starting with the moduli space for $N$-marked supersingular K3 surfaces. Similar to the construction in the complex case by Namikawa \cite{MR771979}, we regard Enriques surfaces as equivalence classes of certain embeddings of $\Gamma(2)$ into the Néron-Severi lattice of a K3 surface. Over the complex numbers this means that Namikawa obtains the moduli space of Enriques surfaces by taking a certain open subscheme of the moduli space of K3 surfaces and then taking the quotient by a group action. 

However, the supersingular case is more complicated than the situation over the complex numbers. One of the main problems we face is the fact that in our situation we are, morally speaking, dealing with several moduli spaces $\CS_i$ nested in each other, with group actions on these subspaces. We use different techniques from \cite{MR3572553} and \cite{MR3084720} concerning pushouts of algebraic spaces and quotients of algebraic spaces by group actions, and finally obtain the following result.
\begin{theorem*}[Theorem \ref{main}] The functor $\underline{\CE}_{\sigma}$ is represented by a scheme $\gls{E_s}$ which is locally of finite type over $\mathbb{F}_p$ and there exists an étale surjective morphism $\pi^E_{\sigma} \colon \CE_{\sigma} \rightarrow \CQ_{\sigma}$ to a separated AF finite type $\mathbb{F}_p$-scheme $\CQ_{\sigma}$.  \end{theorem*}
Here, a scheme $X$ is called AF, if every finite subset of $X$ is contained in an affine open subscheme of $X$.

The geometry of the scheme $\CE_{\sigma}$ is complicated in general, but we have some results on the number of its connected and irreducible components. In short, these numbers depend on properties of the lattice $N_{\sigma}$ and we refer to Section \ref{geom} for details.

Since the scheme $\CQ_{\sigma}$ in the theorem above was constructed from the scheme $\CM_{\sigma}$, we also obtain a Torelli theorem for Enriques quotients of supersingular K3 surfaces.

\begin{theorem*}[Theorem \ref{torelli}] Let $Y_1$ and $Y_2$ be supersingular Enriques surfaces. Then $Y_1$ and $Y_2$ are isomorphic if and only if  $\pi_{\sigma}^E(Y_1)=\pi_{\sigma}^E(Y_2)$ for some $\sigma \leq 5$. \end{theorem*}

 The period map $\pi_{\sigma}^E$ is defined in the following way: the scheme $\CQ_{\sigma}$  represents the functor that associates to a smooth scheme $S$ the set of isomorphism classes of families of K3 crystals $H$ over $S$ together with maps $\gamma \colon \Gamma(2) \hookrightarrow T_H \hookrightarrow H$ that are compatible with intersection forms and such that there exists a factorization $\gamma \colon \Gamma(2) \hookrightarrow N_{\sigma} \hookrightarrow T_H \hookrightarrow H$ without $(-2)$-vectors in the orthogonal complement $\gamma(\Gamma(2))^{\perp} \subset N_{\sigma}$. For a supersingular Enriques surface $Y$ we can choose a $\Gamma$-marking $\gamma \colon \Gamma \rightarrow \mathrm{NS}(Y)$ and this induces a point $\pi_{\sigma}^E(Y, \gamma) \in \CQ_{\sigma}$. We show that $\pi_{\sigma}(Y, \gamma)$ is independent of the choice of $\gamma$ and set $\pi_{\sigma}^E(Y)= \pi_{\sigma}^E(Y, \gamma)$. This construction justifies calling $\pi_{\sigma}^E(Y)$ the \emph{period of $Y$} and we call $\CQ_{\sigma}$ the \emph{period space} of supersingular Enriques surfaces of Artin invariant at most $\sigma$.
 
It remains to mention characteristic p=2. Here, there are three types of Enriques surfaces and a moduli space in this case has two components \cite{MR491720} \cite{MR3393362}. For the component corresponding to simply-connected Enriques surfaces, Ekedahl, Hyland and Shepherd-Barron \cite{ekedahl2012moduli} constructed a period map and established a Torelli theorem. In their work, however, the K3-like cover is not smooth and the covering is not étale, which is why the theory there has a slightly different flavor. 
\section*{Funding}  This work was supported by the ERC Consolidator Grant 681838 “K3CRYSTAL”.
\section*{Acknowledgements} I thank my doctoral advisor Christian Liedtke for his extensive support of my work. I am indebted to David Rydh for giving me access to an unpublished preprint, pointing to another well-hidden preprint of his and taking the time to explain some details to me, all concerning pushouts of Deligne-Mumford stacks. I also thank Paul Hamacher for many helpful discussions. Further thanks go to Roberto Laface and Gebhard Martin for helpful comments on an earlier version of this article. 

\section{Prerequisites and notation}\label{overview}
In this section we fix some notation and recall known results on supersingular K3 surfaces.

Let $k$ be an algebraically closed field of characteristic $p \geq 3$. A K3 surface $X$ over $k$ is called \emph{supersingular} if and only if $\mathrm{rk}(\mathrm{NS}(X)) = 22$. This definition of supersingularity is due to Shioda. There is a second definition for supersingularity due to Artin. Namely, a K3 surface $X$ over $k$ is called \emph{Artin supersingular} if and only if its formal Brauer group $\Phi^2_X$ is of infinite height. It follows from the Tate conjecture that, over any perfect field $k$, a K3 surface is Artin supersingular if and only if it is Shioda supersingular \cite{MR3265555}. Charles first proved the Tate conjecture over fields of characteristic at least $5$ \cite{MR3103257}. Using the Kuga-Satake construction, Madapusi Pera gave a proof of the Tate conjecture over fields of characteristic at least $3$ \cite{MR3370622}. Over fields of characteristic $p=2$, the Tate conjecture was proved by Kim and Madapusi Pera \cite{MR3569319}. By a \emph{lattice} $(L, \langle\cdot,\cdot\rangle)$ we mean a free $\mathbb{Z}$-module $L$ of finite rank together with a nondegenerate symmetric bilinear form $\langle \cdot, \cdot \rangle \colon L \times L \rightarrow \mathbb{Z}$.

Most of the following content is due to Ogus \cite{MR563467}\cite{MR717616}. A strong inspiration for our treatment in this section and a good source for the interested reader is \cite{MR3524169}.
\subsection{K3 crystals}
For the definition of $F$-crystals and their slopes we refer to \cite[Chapter I.1]{MR563463}. Given a supersingular K3 surface $X$, it turns out that a lot of information is encoded in its second crystalline cohomology. We say that $H^2_{\mathrm{crys}}(X/W)$ is a \emph{supersingular K3 crystal} of rank $22$ in the sense of the following definition, due to Ogus \cite{MR563467}.
\begin{definition} Let $k$ be a perfect field of positive characteristic $p$ and let $W=W(k)$ be its Witt ring with lift of Frobenius $\sigma \colon W \rightarrow W$. A \emph{supersingular K3 crystal of rank $n$} over $k$ is a free $W$-module $H$ of rank $n$ together with an injective  $\sigma$-linear map 
\begin{align*} \varphi \colon H \rightarrow H, \end{align*}
 i.e.\ $\varphi(a \cdot m)= \sigma(a) \cdot \varphi(m)$ for all $a \in W$ and $m \in H$, and a symmetric bilinear form
 \begin{align*}
 \langle -,- \rangle \colon H \times H \rightarrow W, 
 \end{align*}
such that
\begin{enumerate}
\item $p^2H \subseteq \mathrm{im}(\varphi)$,
\item the map $\varphi \otimes_W k$ is of rank $1$,
\item $\langle -,- \rangle$ is a perfect pairing,
\item $\langle \varphi(x), \varphi(y) \rangle= p^2\sigma \left(\langle x,y \rangle \right)$, and
\item the $F$-crystal $(H,\varphi)$ is purely of slope $1$.
\end{enumerate}
\end{definition}
The \emph{Tate module} $T_H$ of a K3 crystal $H$ is the $\mathbb{Z}_p$-module 
$$ T_H \coloneqq \{x \in H \mid \varphi(x)= px\}.$$
One can show that if $H=H^2_{\mathrm{crys}}(X/W)$ is the second crystalline cohomology of a supersingular K3 surface $X$ and $c_1 \colon \mathrm{Pic}(X) \rightarrow H^2_{\mathrm{crys}}(X/W)$ is the first crystalline Chern class map, we have $c_1(\mathrm{Pic}(X)) \subseteq T_H$. If $X$ is defined over a finite field, the Tate conjecture is known, see \cite{MR3103257} \cite{MR3370622}, and it follows that we even have the equality $c_1(\mathrm{NS}(X)) \otimes \mathbb{Z}_p= T_H$. The following proposition on the structure of the Tate module of a supersingular K3 crystal is due to Ogus \cite{MR563467}.
\begin{proposition}
Let $(H,\varphi, \langle -,- \rangle)$ be a supersingular K3 crystal and let $T_H$ be its Tate module. Then $\mathrm{rk}_W H= \mathrm{rk}_{\mathbb{Z}_p} T_H$ and the bilinear form $(H, \langle -,- \rangle)$ induces a non-degenerate form $T_H \times T_H \rightarrow \mathbb{Z}_p$ via restriction to $T_H$ which is not perfect. More precisely, we find
\begin{enumerate}
\item $\mathrm{ord}_p(T_H)=2 \sigma$ for some positive integer $\sigma$,
\item $(T_H, \langle -,-\rangle)$ is determined up to isometry by $\sigma$,
\item $\mathrm{rk}_W H \geq 2 \sigma$ and
\item there exists an orthogonal decomposition
\begin{align*}
(T_H, \langle -,- \rangle) \cong (T_0,p \langle -,- \rangle) \perp (T_1, \langle -,- \rangle),
\end{align*}
where $T_0$ and $T_1$ are $\mathbb{Z}_p$-lattices with perfect bilinear forms and of ranks $\mathrm{rk}T_0= 2 \sigma$ and $\mathrm{rk}T_1= \mathrm{rk}_WH-2 \sigma$.
\end{enumerate}
\end{proposition}
The positive integer $\sigma$ is called the \emph{Artin invariant} of the K3 crystal $H$ \cite{MR563467}. When $H$ is the second crystalline cohomology of a supersingular K3 surface $X$, we have $1 \leq \sigma(H) \leq 10$.
\subsection{K3 lattices}
The previous subsection indicates that the Néron-Severi lattice $\mathrm{NS}(X)$ of a supersingular K3 surface $X$ plays an important role in the study of supersingular K3 surfaces via the first Chern class map. We say that $\mathrm{NS}(X)$ is a \emph{supersingular K3 lattice} in the sense of the following definition due to Ogus \cite{MR563467}.
\begin{definition} A \emph{supersingular K3 lattice} is an even lattice $(N,\langle -,-\rangle)$ of rank $22$ such that
\begin{enumerate}
\item the discriminant $d(N)$ is $-1$ in $\mathbb{Q}^{\ast}/{\mathbb{Q}^{\ast}}^2$,
\item the signature of $N$ is $(1,21)$, and
\item the lattice $N$ is $p$-elementary for some prime number $p$.
\end{enumerate}
\end{definition} 
When $N$ is the Néron-Severi lattice of a supersingular K3 surface $X$, then the prime number $p$ in the previous definition turns out to be the characteristic of the base field. One can show that if $N$ is a supersingular K3 lattice, then its discriminant is of the form $d(N)=-p^{2 \sigma}$ for some integer $\sigma$ such that $1 \leq \sigma \leq 10$. The integer $\sigma$ is  called the \emph{Artin invariant} of the lattice $N$. If $X$ is a supersingular K3 surface, we call $\sigma(\mathrm{NS}(X))$ the \emph{Artin invariant} of the supersingular K3 surface $X$ and we find that $\sigma(\mathrm{NS}(X))=\sigma(H^2_{\mathrm{crys}}(X/W))$. The following theorem is due to Rudakov and Shafarevich \cite[Section 1]{MR633161}.
\begin{theorem} The Artin invariant $\sigma$ determines a supersingular K3 lattice up to isometry.
\end{theorem}
\subsection{Characteristic subspaces and K3 crystals}
In this subsection we introduce characteristic subspaces. These objects yield another way to describe K3 crystals, a little closer to classic linear algebra in flavor. For this subsection we fix a prime $p > 2$ and a perfect field $k$ of characteristic $p$ with Frobenius $F \colon k \rightarrow k$, $x \mapsto x^p$.
\begin{definition} Let $\sigma$ be a non-negative integer and let $V$ be a $2\sigma$-dimensional $\mathbb{F}_p$-vector space. A non-degenerate quadratic form
\begin{align*}
\langle -,- \rangle \colon V \times V \rightarrow \mathbb{F}_p.
\end{align*}
on $V$ is called \emph{non-neutral} if there exists no $\sigma$-dimensional isotropic subspace of $V$.
\end{definition}
\begin{definition}
Let $\sigma$ be a non-negative integer and let $V$ be a $2\sigma$-dimensional $\mathbb{F}_p$-vector space together with a non-degenerate and non-neutral quadratic form
\begin{align*}
\langle -,- \rangle \colon V \times V \rightarrow \mathbb{F}_p.
\end{align*}
Set $\varphi \coloneqq \mathrm{id}_V \otimes F \colon V \otimes k \rightarrow V \otimes k$. A $k$-subspace $G \subset V \otimes k$ is called \emph{characteristic} if
\begin{enumerate}
\item $G$ is a totally isotropic subspace of dimension $\sigma$, and
\item $G + \varphi(G)$ is of dimension $\sigma+1$.
\end{enumerate}
A \emph{strictly characteristic subspace} is a characteristic subspace $G$ such that
\begin{align*}
V \otimes k = \sum_{i=0}^{\infty} \varphi^i(G)
\end{align*}
holds true.
\end{definition}
We can now introduce the categories
\begin{align*}
\mathrm{K}3(k) &\coloneqq \left\{\begin{array}{l}\text{Supersingular K3 crystals over $k$} \end{array} \right\}
\intertext{and}
\mathbb{C}3(k) &\coloneqq \left\{\begin{array}{l}\text{Pairs $(T,G)$, where $T$ is a supersingular} \\ \text{K3 lattice over $\mathbb{Z}_p$, and $G\subseteq T_0 \otimes_{\mathbb{Z}_p} k$} \\
\text{is a strictly characteristic subspace}\end{array} \right\}
\end{align*}
with only isomorphisms as morphisms.
\begin{remark} We note that while $T_0$ and $T_1$ are not functorial in $T$, the quotients $T_0 \otimes \mathbb{F}_p$ and $T_1 \otimes \mathbb{F}_p$ are functorial in $T$ \cite[Remark 3.16]{MR563467}. Hence, the definition of $\mathbb{C}3(k)$ makes sense.
\end{remark}
It turns out that over an algebraically closed field these two categories are equivalent.
\begin{theorem}\cite[Theorem 3.20]{MR563467} Let $k$ be an algebraically closed field of characteristic $p > 2$. Then the functor
\begin{align*}
\mathrm{K}3(k) &\lra \mathbb{C}3(k), \\
\left(H, \varphi, \langle -,- \rangle\right) & \longmapsto \left(T_H, \ker\left(T_H \otimes_{\mathbb{Z}_p} k \rightarrow H \otimes_{\mathbb{Z}_p} k\right) \subset T_0 \otimes_{\mathbb{Z}_p} k\right)
\end{align*}
defines an equivalence of categories.
\end{theorem}
If we denote by $\mathbb{C}3(k)_{\sigma}$ the subcategory of $\mathbb{C}3(k)$ consisting of objects $(T,G)$ where $T$ is a supersingular K3 lattice of Artin invariant $\sigma$, there is a coarse moduli space.
\begin{theorem}\cite[Theorem 3.21]{MR563467} Let $k$ be an algebraically closed field of characteristic $p > 2$. We denote by $\mu_n$ the cyclic group of $n$-th roots of unity. There exists a canonical bijection
\begin{align*}
\left(\mathbb{C}3(k)_{\sigma}/\simeq\right) \lra \mathbb{A}^{\sigma-1}_k(k)/\mu_{p^{\sigma}+1}(k).
\end{align*}

\end{theorem}
The previous theorem concerns characteristic subspaces defined on closed points with algebraically closed residue field. Next, we consider families of characteristic subspaces. 
\begin{definition} Let $\sigma$ be a non-negative integer and let $(V, \langle -,- \rangle)$ be a $2\sigma$-dimensional $\mathbb{F}_p$-vector space together with a non-neutral quadratic form. If $A$ is an $\mathbb{F}_p$-algebra, a direct summand $G \subset V \otimes A$ is called a \emph{geneatrix} if $\mathrm{rk}(G)= \sigma$ and $\langle -,- \rangle$ vanishes when restricted to $G$. A \emph{characteristic geneatrix} is a geneatrix $G$ such that $G + F_A(G)$ is a direct summand of rank $\sigma+ 1$ in $V \otimes A$. We write $\underline{M}_V(A)$ for the set of characteristic geneatrices in $V \otimes A$.
\end{definition}
It turns out that there exists a moduli space for characteristic geneatrices.
\begin{proposition}\cite[Proposition 4.6]{MR563467} The functor
\begin{align*}
(\text{$\mathbb{F}_p$-algebras})^{\mathrm{op}} &\lra (\text{Sets}), \\
A &\longmapsto \underline{M}_V(A)
\end{align*}
is representable by an $\mathbb{F}_p$-scheme $M_V$ which is smooth, projective and of dimension $\sigma-1$.
\end{proposition}
If $N$ is a supersingular K3 lattice with Artin invariant $\sigma$, then $N_0= pN^{\vee}/pN$ is a $2\sigma$-dimensional $\mathbb{F}_p$-vector space together with a non-degenerate and non-neutral quadratic form induced from the bilinear form on $N$.
\begin{definition} We set $\CM_{\sigma} \coloneqq M_{N_0}$ and call this scheme the \emph{moduli space of $N$-rigidified K3 crystals}. \end{definition}
\subsection{Ample cones} Next, we will need to enlarge $\CM_{\sigma}$ by equipping $N$-rigidified K3 crystals with ample cones. For the rest of this section we fix a prime $p \geq 3$.
\begin{definition} 
Let $N$ be a supersingular K3 lattice. The set $\Delta_N \coloneqq \{l \in N \mid l^2=-2\}$ is called the set of \emph{roots of $N$}. The \emph{Weyl group $W_N$} of $N$ is the subgroup of the orthogonal group $O(N)$ generated by all automorphisms of the form $s_l \colon x \mapsto x + \langle x,l \rangle l$ with $l \in \Delta_N$. We denote by $\pm W_N$ the subgroup of $O(N)$ generated by $W_N$ and $\pm \mathrm{id}$. Further, we define
\begin{align*}
V_N \coloneqq \{x \in N \otimes \mathbb{R} \mid x^2 > 0 \text{ and } \langle x,l \rangle \neq 0 \text{ for all } l \in \Delta_N\}.
\end{align*}
The set $V_N$ is an open subset of $N \otimes \mathbb{R}$ and each of its connected components meets $N$. The connected components of $V_N$ are called the \emph{ample cones of $N$} and we denote by $C_N$ the set of ample cones of $N$.
\end{definition}
\begin{remark}
The group $\pm W_N$ operates simply and transitively on $C_N$ \cite{MR717616}.
\end{remark}
\begin{definition}
Let $N$ be a supersingular K3 lattice of Artin invariant $\sigma$ and let $S$ be an algebraic space over $\mathbb{F}_p$. For a characteristic geneatrix $G \in \CM_{\sigma}(S)$ and any point $s \in S$ we define
\begin{align*}
\Lambda(s) &\coloneqq N_0 \cap G(s), \\
N(s) &\coloneqq \{x \in N \otimes \mathbb{Q} \mid px \in N \text{ and } \overline{px} \in \Lambda(s)\}, \\
\Delta(s) &\coloneqq \{l \in N(s) \mid l^2=-2\}.
\end{align*}
An \emph{ample cone for $G$} is an element $\alpha \in \prod_{s \in S} C_{N(s)}$ such that $\alpha(s) \subseteq \alpha(t)$ whenever $s \in \overline{\{t\}}$.
\end{definition}
\section{Moduli spaces of \texorpdfstring{$N_{\sigma}$-marked}{marked} supersingular K3 surfaces}\label{secclasmark}
This section discusses the moduli spaces for lattice-marked K3 surfaces that were introduced in \cite{MR717616}.
 
We fix a prime $p \geq 3$ and for each integer $\sigma$ with $1 \leq \sigma \leq 10$ we fix a K3 lattice $N_{\sigma}$ with $\sigma(N_{\sigma})=\sigma$. A \textit{family of supersingular K3 surfaces} is a smooth and proper morphism $f \colon \CX \rightarrow S$  of algebraic spaces over $\mathbb{F}_p$ such that for each field $k$ and each $k$-valued point $\Spec k \rightarrow S$ the fiber $\CX_k \rightarrow \Spec k$ is a projective supersingular K3 surface. By \cite[Theorem 3.1.1]{MR2263236}, \cite[Theorem 2.1]{SB_1961-1962__7__221_0} the relative Picard functor $\underline{\Pic}_{\CX/S}$ is representable by a separated algebraic space $\Pic_{\CX/S}$ over $S$. An \emph{$N_{\sigma}$-marking of a family of supersingular K3 surfaces} $f \colon \CX \rightarrow S$ is a morphism $\psi \colon \underline{N}_{\sigma} \rightarrow \Pic_{\CX/S}$ of group objects in the category of algebraic spaces that is compatible with intersection forms. There is an obvious notion of morphisms of families $N_{\sigma}$-marked K3 surfaces. From now on we will write \gls{A_p} for the category of algebraic spaces over $\mathbb{F}_p$. We consider the following moduli problem
\begin{align*} \underline{\CS}_{\sigma} \colon \CA_{\mathbb{F}_p}^{\mathrm{op}} &\lra \left(\text{Sets} \right) \\
S &\longmapsto \left\{\begin{array}{l}
\text{Isomorphism classes of tuples $(f, \psi)$ such that} \\ 
\text{$f \colon \CX \rightarrow S$ is a family of supersingular K3 surfaces} \\
\text{and $\psi \colon \underline{N}_{\sigma} \rightarrow \Pic_{\CX/S}$ is an $N_{\sigma}$-marking} \end{array} \right\}.\end{align*}
It is a classical result of Ogus that the functor $\underline{\CS}_{\sigma}$ is representable by an $\mathbb{F}_p$-scheme $\gls{S_s}$ that is smooth of dimension $\sigma-1$ and locally of finite type over $\mathbb{F}_p$ \cite{MR717616}. Further, $\CS_{\sigma}$ satisfies the valuative criterion for universal closedness. However, $\CS_{\sigma}$ is in general neither quasi-compact nor separated. 

Via the period map the functor $\underline{\CS}_{\sigma}$ is canonically isomorphic to a functor $\underline{\CP}_{\sigma}$ \cite{MR717616} which is defined to be
\begin{align*} \underline{\CP}_{\sigma} \colon  \CA_{\mathbb{F}_p}^{\mathrm{op}} &\lra \left(\text{Sets} \right) \\
S &\longmapsto \left\{\begin{array}{l}\text{characteristic generatrices $G \subseteq pN_{\sigma}^{\vee}/pN_{\sigma} \otimes \CO_S$} \\ \text{together with an ample cone} \end{array} \right\}.  \end{align*}
Ogus originally proved that the period morphism $\pi \colon \CS_{\sigma} \lra \CP_{\sigma}$ is an isomorphism over fields of characteristic at least $5$, but Bragg and Lieblich recently showed that Ogus' results also hold true in characteristic $3$ \cite[Section 5.1]{2018arXiv180407282B}.

If we consider the functor
\begin{align*} \underline{\CM}_{\sigma} \colon \CA_{\mathbb{F}_p}^{\mathrm{op}} &\lra \left(\text{Sets} \right) \\
S &\longmapsto \left\{\text{characteristic generatrices $G \subseteq pN_{\sigma}^{\vee}/pN_{\sigma} \otimes \CO_S$}\right\}, \end{align*}
then there is a canonical surjection of functors $\pi_{\sigma} \colon \underline{\CS}_{\sigma} \rightarrow \underline{\CM}_{\sigma}$ which is given by forgetting the choice of an ample cone. The functor $\underline{\CM}_{\sigma}$ is representable by a smooth connected projective scheme $\gls{M_s}$ of dimension $\sigma-1$ and the morphism of schemes $\pi_{\sigma}$ is étale. For further details on the functor $\underline{\CM}_{\sigma}$ we refer the interested reader to \cite{MR563467} and for further details on the functor $\underline{\CS}_{\sigma}$ we refer to \cite{MR717616}. 

Now let $\sigma' < \sigma$ be positive integers with $\sigma \leq 10$. In our construction of the moduli space of marked Enriques surfaces we will use an inductive argument. Therefore, we begin with an observation on the relation between the schemes $\CS_{\sigma}$ and $\CS_{\sigma'}$. There exists an embedding of lattices $j \colon N_{\sigma} \hookrightarrow N_{\sigma'}$ which makes $N_{\sigma'}$ into an overlattice of $N_{\sigma}$. We say that two such embeddings $j_1$ and $j_2$ are \emph{isomorphic} embeddings if there exists an automorphism \mbox{$\alpha \colon N_{\sigma'} \rightarrow N_{\sigma'}$} such that $\alpha \circ j_1 = j_2$. 

By \cite[Proposition 1.4.1]{1980IzMat..14..103N} there are only finitely many isomorphism classes of such  embeddings $j \colon N_{\sigma} \hookrightarrow N_{\sigma'}$. For each isomorphism class we choose a representative $j$ and denote by \gls{R_ss} the set of these representatives. An embedding $j \colon N_{\sigma} \hookrightarrow N_{\sigma'}$ induces a morphism of $\mathbb{F}_p$-schemes
\begin{align*} \Phi_j \colon \CS_{\sigma'} &\lra \CS_{\sigma}  
\intertext{by mapping} 
(f \colon \CX \rightarrow S, \psi \colon \underline{N}_{\sigma'} \rightarrow \Pic_{\CX/S}) &\longmapsto (f \colon \CX \rightarrow S, \psi \circ j \colon \underline{N}_{\sigma}\rightarrow \Pic_{\CX/S}) \end{align*}
on $S$-valued points. Similarly, we also obtain a morphism $\Psi_j \colon \CM_{\sigma'} \rightarrow \CM_{\sigma}$. It follows from \cite[Remark 4.8]{MR563467} that the $\Psi_j$ are closed immersions. Analogously, we see that the finite union $\gls{M_ss} = \bigcup_{j \in R_{\sigma',\sigma}} \Psi_j(\CM_{\sigma'})$ is the closed subscheme in $\CM_{\sigma}$ corresponding to characteristic subspaces $G$ of $N_{\sigma}$ with Artin invariant $\sigma(G) \leq \sigma'$. We now want to show that the morphisms $\Phi_j$ are also closed immersions.
\begin{lemma}\label{cartphi} The commutative diagrams
\begin{align*} \xymatrix{\CS_{\sigma'} \ar[r]^{\Phi_j} \ar[d]_{\pi_{\sigma'}} & \CS_{\sigma} \ar[d]^{\pi_{\sigma}}\\
   \CM_{\sigma'} \ar[r]_{\Psi_j}& \CM_{\sigma}}\end{align*}
are cartesian.
\end{lemma}
\begin{proof}
It is easy to see that the $\Phi_j$ are monomorphisms of functors. So we only need to check the existence part in the definition of fiber products. To this end, we claim that there is an equality $\Phi_j(\CS_{\sigma'})= \pi_{\sigma}^{-1}(\Psi_j(\CM_{\sigma'}))$. Indeed, the inclusion $\Phi_j(\CS_{\sigma'}) \subseteq \pi_{\sigma}^{-1}(\Psi_j(\CM_{\sigma'}))$ is clear by definition and we easily see that the two subschemes have the same underlying topological space, that is, we have an equality of sets $\{x \in \pi_{\sigma}^{-1}(\Psi_j(\CM_{\sigma}))\} = \{x \in \Phi_j(\CS_{\sigma'})\}$. The scheme $\pi_{\sigma'}^{-1}(\Psi_j(\CM_{\sigma'}))$ is reduced since $\pi_{\sigma}$ is an étale morphism and $\Psi_j(\CM_{\sigma'})$ is reduced. Hence, we obtain the desired equality of subschemes. 

Thus, given an $\mathbb{F}_p$-scheme $S$ and $S$-valued points $y \in \CM_{\sigma'}(S)$ and $z \in \CS_{\sigma}(S)$ such that $\Psi_j(y)= \pi_{\sigma}(z)$, we find that $z \in \Phi_j(\CS_{\sigma'}(S))$. If we let $x$ be the preimage of $z$ under $\Phi_j(S)$, then $\Phi_j(x)=z$ and $\pi_{\sigma'}(x)=y$ which shows the claim.
\end{proof}
\begin{proposition} The morphisms of functors $\Phi_j \colon \underline{\CS}_{\sigma'} \rightarrow \underline{\CS}_{\sigma}$ are closed immersions of functors and the subfunctor $\underline{\CS}_{\sigma}^{\sigma'} \hookrightarrow \underline{\CS}_{\sigma}$ which is defined to be
\begin{align*} \underline{\CS}_{\sigma}^{\sigma'} \colon  \CA_{\mathbb{F}_p}^{\mathrm{op}} &\lra \left(\text{Sets} \right) \\
S &\longmapsto \left\{\begin{array}{l}\text{Isomorphism classes of tuples $(f, \psi)$ such that} \\ 
\text{$f \colon \CX \rightarrow S$ is a family of supersingular K3 surfaces} \\
\text{and $\psi \colon \underline{N}_{\sigma} \rightarrow \Pic_{\CX/S}$ is an $N_{\sigma}$-marking} \\
\text{such that each fiber $\CX_s$ has $\sigma(\CX_s) \leq \sigma'$}\end{array} \right\} \end{align*}
is representable by the closed subscheme $\gls{S_ss}= \bigcup_{j \in R_{\sigma', \sigma}} \Phi_j(\CS_{\sigma'}) \subseteq \CS_{\sigma}$. \end{proposition}
\begin{proof} 
We already mentioned that the morphisms $\Psi_j$ are closed immersions, and thus Lemma \ref{cartphi} implies that the morphisms $\Phi_j$ are closed immersions as well. The assertion on the functor represented by the union $\bigcup_{j \in R_{\sigma',\sigma}} \Phi_j(\CS_{\sigma'})$ is a consequence of the equality 
\begin{align*} \bigcup_{j \in R_{\sigma',\sigma}} \Phi_j(\CS_{\sigma'}) = \pi^{-1}\left(\bigcup_{j \in R_{\sigma',\sigma}} \Psi_j(\CM_{\sigma'})\right), \end{align*} 
which follows from the proof of Lemma \ref{cartphi}.  \end{proof}
\section{Auxiliary functors and moduli spaces}\label{aux}
In this section we will introduce some auxiliary functors which we will then use to construct the main functor in the subsequent section.

In the following, we let $\sigma \leq 10$ be a positive integer. We consider the lattice \mbox{$\gls{Gamma} = U_2 \oplus E_8(-1)$}, which is up to isomorphism the unique unimodular, even lattice of signature $(1,9)$. The Picard group of any Enriques surface is isomorphic to $\Gamma \oplus \mathbb{Z}/2\mathbb{Z}$. Our idea is as follows: if $Y$ is an Enriques surface with a supersingular covering K3 surface $X$, then we can see the quotient map $f \colon X \rightarrow Y$ as a primitive embedding of lattices $\gamma \colon \Gamma(2) \hookrightarrow \mathrm{NS}(X)$ such that $\Gamma(2)$ contains an ample divisor and such that there is no $(-2)$-vector in $\gamma(\Gamma(2))^{\perp} \subseteq \mathrm{NS}(X)$ \cite{2013arXiv1301.1118J}. If we also admit embeddings $\gamma \colon \Gamma(2) \hookrightarrow \mathrm{NS}(X)$ such that there is a $(-2)$-vector in $\gamma(\Gamma(2))^{\perp} \subset \mathrm{NS}(X)$, then we talk about quotients $X \rightarrow Y'$ of $X$ by an involution that maybe has a non-trivial fixed point locus.

We will therefore define various functors of $\Gamma(2)$-marked K3 surfaces and in Section \ref{enr} we then show that the main functor $\tilde{\CE}_{\sigma}$ of $\Gamma(2)$-marked K3 surfaces from Section \ref{secmark} is isomorphic to a functor of $\Gamma$-marked Enriques surfaces.
 \subsection{The functor $\CS'_{\gamma}$}
By Corollary 2.4. in \cite{MR3350105}, there exists a primitive embedding of lattices $\gamma \colon \Gamma(2) \hookrightarrow N_{\sigma}$ such that $\gamma(\Gamma(2))^{\perp} \subset N_{\sigma}$ contains no vector of self-intersection number $-2$ if and only if $\sigma \leq 5$, and further there are only finitely many isomorphism classes $[\gamma \colon \Gamma(2) \hookrightarrow N_{\sigma}]$ of such embeddings (where again two embeddings $\gamma_1 \colon \Gamma(2) \hookrightarrow N_{\sigma}$ and $\gamma_2 \colon \Gamma(2) \hookrightarrow N_{\sigma}$ are called isomorphic if there exists an automorphism $\alpha$ of $N_{\sigma}$ such that $\alpha \circ \gamma_1 = \gamma_2$). We fix for each such isomorphism class a representative $\gamma$ and denote by \gls{R_s} the finite set formed by these elements. For $\sigma > 5$ we have $R_{\sigma} = \emptyset$. For an embedding $\gamma \in R_{\sigma}$ we consider the subfunctor $\underline{\CS}'_{\gamma} \subset \underline{\CS}_{\sigma}$ which is defined to be
\begin{align*} \underline{\CS}'_{\gamma}\colon \CA_{\mathbb{F}_p}^{\mathrm{op}} &\lra \left(\text{Sets} \right) \\
S &\longmapsto \left\{\begin{array}{l}\text{Isomorphism classes of tuples $(f, \psi)$ such that} \\ 
\text{$f \colon \CX \rightarrow S$ is a family of supersingular K3 surfaces} \\
\text{and $\psi \colon \underline{N}_{\sigma} \rightarrow \Pic_{\CX/S}$ is an $N_{\sigma}$-marking}\\
\text{such that for each geometric point $s \in S$} \\
\text{the sublattice $\gamma_s(\Gamma(2)) \hookrightarrow \mathrm{NS}(\CX_s)$ } \\
\text{$\gamma'_s(\Gamma(2))^{\perp} \hookrightarrow \mathrm{NS}(\CX_s)$ contains no $(-2)$-vector}\end{array} \right\}.  \end{align*}
We are interested in the representability of the functor $\underline{\CS}'_{\gamma}$.
\begin{definition} We write $R_{\gamma}$ for the set of representatives of all isomorphism classes of embeddings $j \colon N_{\sigma} \hookrightarrow N_{\sigma'}$ such that $j \left( \gamma(\Gamma(2))\right)^{\perp} \subseteq N_{\sigma'}$ contains a $(-2)$-vector.
\end{definition} 
\begin{proposition}\label{open1} The functor $\underline{\CS}'_{\gamma}$ is representable by an open subscheme $\gls{S'_g}$ of $\CS_{\sigma}$. \end{proposition}
\begin{proof} The set $R_{\gamma}$ is a subset of the finite set $\bigcup_{\sigma' < \sigma} R_{\sigma',\sigma}$. For each $j$, the subscheme $\Phi_j(\CS_{\sigma'}) \subseteq \CS_{\sigma}$ is closed, and it is clear that the open subscheme 
\begin{align*} \CS'_{\gamma} = \CS'_{\gamma} \backslash \left(\bigcup_{j \in R_{\gamma}} \Phi_j(\CS_{\sigma'})\right) \end{align*}
represents the functor $\underline{\CS}'_{\gamma}$. \end{proof}
We also find an open subscheme of $\CM_{\sigma}$ that lies under $\CS'_{\gamma}$. Namely, we define $\CM'_{\gamma}$ as the image of $\CS'_{\gamma}$ under $\pi_{\sigma}$.
\begin{proposition} The $\mathbb{F}_p$-scheme $\gls{M'_g}$ is quasi-projective and the canonical surjective morphism of algebraic spaces $\pi'_{\gamma} \colon \CS'_{\gamma} \rightarrow \CM'_{\gamma}$ is étale. \end{proposition}
\begin{proof} This is clear since the morphism $\pi_{\sigma} \colon \CS_{\sigma} \rightarrow \CM_{\sigma}$ is universally open. \end{proof}
\begin{remark} It is not clear to us whether the functors $\underline{\CS}'_{\gamma}$ and $\CS_{\sigma}$ are equal in general. However, we think this should not be true. The lattice theoretic question we have to answer is 
\begin{center}
\emph{Do there exist embeddings $j \colon N_{\sigma} \hookrightarrow N_{\sigma'}$ and $\gamma \colon \Gamma(2) \hookrightarrow N_{\sigma}$ such that \\
$\gamma(\Gamma(2))^{\perp}$ contains no $-2$-vector, but $j( \gamma(\Gamma(2)))^{\perp}$ contains a $-2$-vector?}
\end{center}
Assuming the answer to this question is \emph{yes}, we could see Proposition \ref{open1} as a supersingular analogue to the fact that the period map of Enriques surfaces in characteristic zero maps to a quotient of the moduli space of K3 surfaces minus a divisor \cite[Theorem 1.14]{MR771979}. We removed a divisor or the empty set in each sub moduli space $\CS_{\sigma'} \subseteq \CS_{\sigma}$. \end{remark}
\subsection{The functor $\CS''_{\gamma}$}
We next consider the subfunctor $\underline{\CS}''_{\gamma}$ of $\underline{\CS}'_{\gamma}$ that only allows markings which contain an ample line bundle in every fiber. That means
\begin{align*} \underline{\CS}''_{\gamma}\colon \CA_{\mathbb{F}_p}^{\mathrm{op}} &\lra \left(\text{Sets} \right) \\
S &\longmapsto \left\{\begin{array}{l}\text{Isomorphism classes of tuples $(f, \psi)$ such that} \\ 
\text{$f \colon \CX \rightarrow S$ is a family of supersingular K3 surfaces} \\
\text{and $\psi \colon \underline{N}_{\sigma} \rightarrow \Pic_{\CX/S})$ is an $N_{\sigma}$-marking}\\
\text{such that for each geometric point $s \in S$} \\
\text{the sublattice $\gamma_s(\Gamma(2)) \hookrightarrow \mathrm{NS}(\CX_s)$ } \\
\text{contains an ample line bundle} \\
\text{and $\gamma'_s(\Gamma(2))^{\perp} \hookrightarrow \mathrm{NS}(\CX_s)$ contains no $(-2)$-vector}\end{array} \right\}.  \end{align*}
It follows from the following lattice theoretic lemma that the induced embedding of lattices $\gamma_s \colon \Gamma(2) \hookrightarrow \mathrm{NS}(\CX_s)$ is primitive even on the locus where the $N_{\sigma}$-marking $\psi \colon \underline{N}_{\sigma} \rightarrow \Pic_{\CX/S}$ is not an isomorphism.
\begin{lemma}\label{primi} Let $\gamma \colon \Gamma(2) \hookrightarrow N_{\sigma}$ be a primitive embedding and let $j \colon N_{\sigma} \hookrightarrow N_{\sigma-1}$ be an embedding of K3 lattices. Then the composition $j \circ \gamma \colon \Gamma(2) \hookrightarrow N_{\sigma-1}$ is a primitive embedding. \end{lemma}
\begin{proof} We write $\Gamma(2)^{\mathrm{sat}}$ for the saturation of $\Gamma(2)$ in $N_{\sigma-1}$. Then we have an inclusion $2 \cdot \Gamma(2)^{\mathrm{sat}} \subset \Gamma(2)$, because the lattice $\Gamma(2)$ is $2$-elementary. On the other hand, we find that $N_{\sigma} + \Gamma(2)^{\mathrm{sat}}$ is an overlattice of $N_{\sigma}$ with $2 \cdot (N_{\sigma} + \Gamma(2)^{\mathrm{sat}}) \subset N_{\sigma}$. Since the lattice $N_{\sigma}$ is $p$-elementary and we have $p \neq 2$, it follows that $N_{\sigma} + \Gamma(2)^{\mathrm{sat}} = N_{\sigma}$. Thus we have an equality $\Gamma(2)=\Gamma(2)^{\mathrm{sat}}$.\end{proof} 
For the rest of the discussion, we will always assume an embedding of $\Gamma(2)$ into some lattice to be primitive. The next thing we are interested in, is the representability of the functor $\underline{\CS}''_{\gamma}$ for some fixed $\gamma \in R_{\sigma}$. The following result is probably known to experts, but we
report it for convenience to the reader.
\begin{proposition}\label{amp} The functor $\underline{\CS}''_{\gamma}$ is an open subfunctor of $\underline{\CS}'_{\gamma}$. \end{proposition}
\begin{proof} By definition, we have to show that for any $\mathbb{F}_p$-scheme $S$ and any isomorphism class 
$x=(f \colon \CX \rightarrow S, \psi \colon \underline{N}_{\sigma} \hookrightarrow \Pic_{\CX/S}) \in \underline{\CS}'_{\gamma}(S)$
 the locus $S_x \subseteq S$ such that $\gamma_s(\Gamma(2))$ contains an ample line bundle for all geometric points $s \in S_x$ is an open subscheme of $S$. 

Given an $\mathbb{F}_p$-scheme $S$ and an $S$-valued point $x=(f \colon \CX \rightarrow S, \psi \colon \underline{N}_{\sigma} \hookrightarrow \Pic_{\CX/S}) \in \underline{\CS}'_{\gamma}(S)$, using Lemma \ref{primi}, we obtain a unique involution $\iota^{\ast}_{\gamma} \colon \Pic_{\CX/S} \rightarrow \Pic_{\CX/S}$ which is induced from $\iota^{\ast}_{\gamma}|_{\Gamma(2)}= \mathrm{id}_{\Gamma(2)}$ and $\iota^{\ast}_{\gamma}|_{\Gamma(2)^{\perp}} = -\mathrm{id}_{\Gamma(2)^{\perp}}$, cf.\ \cite[Proposition 2.1.1.]{MR2452829}. By Ogus' Torelli theorem \cite{MR717616} and the argument in \cite[Lemma 4.3.]{2013arXiv1301.1118J}, the automorphism $\iota^{\ast}_{\gamma}$ is induced from an automorphism of $S$-algebraic spaces $\iota_{\gamma} \colon \CX \rightarrow \CX$ if and only if $\gamma(\Gamma(2)) \hookrightarrow \Pic_{\CX/S}$ intersects the ample cone in $\mathrm{NS}(\CX_s)$ for all points $s \in S$. 

Now, if there is no point $s \in S$ such that $\gamma_s(\Gamma(2)) \hookrightarrow \mathrm{NS}(\CX_s)$ contains an ample line bundle, then $S_x = \emptyset$ is the empty scheme, which is an open subscheme of $S$. Else, let $s \in S$ be a point such that $\gamma_s(\Gamma(2)) \hookrightarrow \mathrm{NS}(\CX_s)$ contains an ample line bundle. Let $\CO_{S,s}$ be the local ring of $S$ at $s$, then $(f \colon \CX_{\Spec \CO_{S,s}} \rightarrow \Spec \CO_{S,s}, \psi \colon N_{\sigma} \hookrightarrow \Pic_{\CX_{\Spec \CO_{S,s}}/\Spec \CO_{S,s}})$ is also an an element of $\underline{\CS}''_{\gamma}(\Spec \CO_{S,s})$ by the discussion in \cite[pages 373-374]{MR717616}. If $\{U_i\}_{i \in I}$ is the directed system of all open subschemes of $S$ such that $s \in U_i$, then $\Spec \CO_{S,s} = \mathrm{lim} U_i$ and we consider the commutative diagram
\begin{align*} \xymatrix{\mathrm{colim}\left(\Aut_{U_i}(\CX_{U_i})\right) \ar[r] \ar[d]_{\cong} & \mathrm{colim} \left(\Aut(\Pic_{\CX_{U_i}/U_i}) \right) \ar[d]_{\cong}\\
   \Aut_{\Spec \CO_{S,s}}(\CX_{\Spec \CO_{S,s}}) \ar[r]& \Aut(\Pic_{\CX_{\Spec \CO_{S,s}}/\Spec \CO_{S,s}}).}\end{align*}
The morphisms $\CX \rightarrow S$ and $\Pic_{\CX/S} \rightarrow S$ are locally of finite presentation, and it follows from \cite[Proposition 31.6.1.]{stacks-project} that the vertical arrows in the diagram are isomorphisms. Further, the horizontal arrows are injective by the Torelli theorem \cite{MR717616} and the fact that filtered colimits of sets preserve injections. Since the automorphism $\iota^{\ast}_{\gamma}|_{\Spec \CO_{S,s}}$ is induced from an automorphism $\iota \in  \Aut_{\Spec \CO_{S,s}}(\CX_{\Spec \CO_{S,s}})$ it follows that there exists an open neighborhood $U(s)$ of $s$ such that $\iota^{\ast}_{\gamma}|_{U(s)}$ is induced from an automorphism $\iota \in \Aut(\Pic_{\CX_{U^s}/U(s)})$. 

Thus, the sublattice $\gamma_{\tilde{s}}(\Gamma(2)) \hookrightarrow \mathrm{NS}(\CX_{\tilde{s}})$ contains an ample line bundle for all $\tilde{s} \in U(s)$. Now let $A$ be the set of all $s \in S$ such that $\gamma_s(\Gamma(2)) \hookrightarrow \mathrm{NS}(\CX_{s})$ contains an ample line bundle. Then $S_x = \bigcup_{s \in A} U(s)$ which is an open subscheme of $S$. \end{proof}
\begin{corollary} The functor $\underline{\CS}''_{\gamma}$ is representable by an open subscheme $\gls{S''_g}$ of $\CS'_{\gamma}$ and the induced morphism $\pi_{\gamma}'' \colon \CS''_{\gamma} \rightarrow \CM'_{\gamma}$ is étale and surjective. \end{corollary}
\begin{proof} The representability is a direct consequence of Proposition \ref{amp}. The morphism $\pi'_{\gamma}$ is étale because being étale is local on the source.

Now, if $s \in \CM'_{\gamma}(k)$ represents a characteristic generatrix $G$ in $pN_{\sigma}^{\vee}/pN_{\sigma} \otimes k$, we can choose an ample cone $\alpha$ for $G$, such that $\gamma_s(\Gamma(2)) \cap \alpha \neq \emptyset$. Indeed, since the signature of $\gamma_s(\Gamma(2))$ is $(1,9)$ and $\gamma_s(\Gamma(2))^{\perp} \cap \Delta = \emptyset$, there exists a $v \in \gamma_s(\Gamma(2)) \cap V_{N_{\sigma}}$. We can then choose $\alpha$ to be the connected component of $V_{N_{\sigma}}$ that contains $v$.

Using the period isomorphism $\CS_{\sigma} \stackrel{\sim}\lra \CP_{\sigma}$, we find a preimage of $s$ in $\CS''_{\gamma}(k)$ from $(G, \alpha) \in \CP_{\sigma}(k)$. Hence $\pi''_{\gamma}$ is surjective. \end{proof}
\subsubsection{The image of $\CS''_{\gamma}$ under the period isomorphism}
It is clear that the image of  $\underline{\CS}''_{\gamma}$ under the period isomorphism $\underline{\CS}_{\sigma} \lra \underline{\CP}_{\sigma}$ is the functor
\begin{align*} {\underline{\CP}''_{\gamma}}^{} \colon \CA_{\mathbb{F}_p}^{\mathrm{op}} &\lra \left(\text{Sets} \right) \\
S &\longmapsto \left\{\begin{array}{l}\text{Tuples $(G,\alpha)$, where} \\ \text{$G \subseteq pN_{\sigma}^{\vee}/pN_{\sigma} \otimes \CO_S$ is a characteristic subspace} \\
\text{and $\alpha \in \prod_{s \in S} C_{N_{\sigma}(s)}$ is the choice of an ample cone} \\
\text{such that $\gamma(\Gamma(2))^{\perp} \cap \Delta_{N_{\sigma}(s)} = \emptyset$} \\
\text{and $\gamma(\Gamma(2)) \cap \alpha(s) \neq \emptyset$ for all $s \in S$} \end{array} \right\} \end{align*}
and $\CM'_{\gamma}$ represents the functor
\begin{align*} {\underline{\CM}'_{\gamma}}^{} \colon \CA_{\mathbb{F}_p}^{\mathrm{op}} &\lra \left(\text{Sets} \right) \\
S &\longmapsto \left\{\begin{array}{l} \text{characteristic generatrices $G \subseteq pN_{\sigma}^{\vee}/pN_{\sigma} \otimes \CO_S$ } \\
\text{such that $\gamma(\Gamma(2))^{\perp} \cap \Delta_{N_{\sigma}(s)} = \emptyset$} \end{array} \right\} \end{align*}
We are interested in an alternative description for an ample cone that meets $\gamma_s(\Gamma(2))$ for each $s \in S$ of an element $G \in{\underline{\CM}'_{\gamma}}^{}(S)$. If $N$ is a supersingular K3 lattice with a $\Gamma(2)$-marking $\gamma$, we write $C_{N}^{\gamma}$ for the set of connected components of $V_N$ that meet $\gamma(\Gamma(2))$.
\begin{lemma}
Let $G \in{\underline{\CM}'_{\gamma}}^{}(S)$ and $s \in S$. There is a natural bijection $c \colon C_{\Gamma(2)} \lra C_{N_{\sigma}(s)}^{\gamma}$.
\end{lemma}
\begin{proof}
Since any $v \in \gamma(\Gamma(2))$ maps to zero in $pN_{\sigma}^{\vee}/pN_{\sigma}$, we find that $\frac{1}{p} \gamma(\Gamma(2))$ is a sublattice of $N_{\sigma}(s)$ for any $s \in S$. Hence, we get an inclusion $V_{\Gamma(2)} \hookrightarrow V_{N_{\sigma}(s)}^{\gamma}$ that meets every connected component of $ V_{N_{\sigma}(s)}^{\gamma}$. This induces a surjection  $C_{\Gamma(2)} \twoheadrightarrow C_{N_{\sigma}(s)}^{\gamma}$. Writing $K$ for the orthogonal complement of $\gamma(\Gamma(2))$ in $N_{\sigma}(s)$, we have $N_{\sigma}(s) \otimes \mathbb{R} = (\gamma(\Gamma(2)) \oplus K) \otimes \mathbb{R}$ and projection to the first component gives a surjection $N_{\sigma}(s) \otimes \mathbb{R} \twoheadrightarrow \Gamma(2) \otimes \mathbb{R}$. We claim that the image of $V_{N_{\sigma}(s)}$ under this map is $V_{\Gamma(2)}$. Indeed, let $v \in V_{\Gamma(2)}$. Then $\gamma(v)^2 = v^2 > 0$. If $w \in \Delta_{N_{\sigma}}(s)$ is an element with $w^2 = -2$, then $w = a(w' + w'')$ with $a \in \mathbb{Q}$, $w' \in \gamma(\Gamma(2))$, $w'' \in K$ and $w' \neq 0$, because $\Delta_{N_{\sigma}}(s) \cap K = \emptyset$. We find $(\gamma(v), w) = (\gamma(v), aw') \neq 0$ and this proves the claim. We therefore get a surjection  $C_{N_{\sigma}(s)}^{\gamma}\twoheadrightarrow C_{\Gamma(2)}$ that is inverse to the previous surjection.
\end{proof}
If $G \in{\underline{\CM}'_{\gamma}}^{}(S)$, we call an element
\begin{align*}
    \alpha \in C_{\Gamma(2)}^{|S|}
\end{align*}
an \emph{ample cone for $G$} if for each pair $s, t \in S$ such that $s$ is a specialization of $t$ we have $c(\alpha(s)) \subseteq c(\alpha(t))$. By the above lemma this is equivalent to our earlier definition of ample cones. This leads us to the following description of ${\underline{\tilde{\CP}}''_{\gamma}}^{}$.
\begin{proposition}[and Definition]\label{amplec}
The functor ${\underline{\tilde{\CP}}''_{\gamma}}^{}$ is canonically isomorphic to the functor
\begin{align*} \CA_{\mathbb{F}_p}^{\mathrm{op}} &\lra \left(\text{Sets} \right) \\
S &\longmapsto \left\{\begin{array}{l}\text{Tuples $(G,\alpha)$, where} \\ \text{$G \subseteq pN_{\sigma}^{\vee}/pN_{\sigma}  \otimes \CO_S$ is a characteristic subspace} \\
\text{such that $\gamma(\Gamma(2))^{\perp} \cap \Delta_{N_{\sigma}(s)} = \emptyset$ for all $s \in S$} \\
\text{and $\alpha \in C_{\Gamma(2)}^{|S|}$ is the choice of an ample cone} \end{array} \right\} \end{align*}
and we identify ${\underline{\tilde{\CP}}''_{\gamma}}^{}$ with this functor.
\end{proposition}

\subsection{The functor $\tilde{\CS}''_{\gamma}$ }
We next want to be able to forget about the choice of a basis for $N_{\sigma}$ in the definition of $\underline{\CS}''_{\gamma}$. To do so, we consider the functor
\begin{align*} {\underline{\tilde{\CS}}''_{\gamma}}^{ps} \colon \CA_{\mathbb{F}_p}^{\mathrm{op}} &\lra \left(\text{Sets} \right) \\
S &\longmapsto \left\{\begin{array}{l}\text{Isomorphism classes of tuples $(f,\underline{\CR}, \gamma')$, where} \\ \text{$f \colon \CX \rightarrow S$ is a family of supersingular K3 surfaces,} \\
\text{$\underline{\CR} \subseteq \Pic_{\CX/S}$ is a subsheaf of lattices and} \\
\text{$\gamma' \colon \underline{\Gamma}(2) \hookrightarrow \underline{\CR}$ is an embedding such that} \\
\text{there exists $\alpha \colon \underline{N}_{\sigma} \cong \underline{\CR}$ with $\alpha \circ \gamma = \gamma'$ and} \\
\text{such that for each geometric point $s \in S$} \\
\text{the sublattice $\gamma'_s(\Gamma(2)) \hookrightarrow \mathrm{NS}(\CX_s)$ } \\
\text{contains an ample line bundle} \\
\text{and $\gamma'_s(\Gamma(2))^{\perp} \hookrightarrow \mathrm{NS}(\CX_s)$ contains no $(-2)$-vector}\end{array} \right\} \end{align*}
and its sheafification with regard to the fppf topology
\begin{align*} {\underline{\tilde{\CS}}''_{\gamma}} \colon \CA_{\mathbb{F}_p}^{\mathrm{op}} &\lra \left(\text{Sets} \right) \\
S &\longmapsto \left\{\begin{array}{l}\text{Isomorphism classes of tuples $(f,\underline{\CR}, \gamma')$, where} \\ \text{$f \colon \CX \rightarrow S$ is a family of supersingular K3 surfaces,} \\
\text{$\underline{\CR} \subseteq \Pic_{\CX/S}$ is a subsheaf of lattices and} \\
\text{$\gamma' \colon \underline{\Gamma}(2) \hookrightarrow \underline{\CR}$ is an embedding such that fppf locally} \\
\text{there exists $\alpha \colon \underline{N}_{\sigma} \cong \underline{\CR}$ with $\alpha \circ \gamma = \gamma'$ and} \\
\text{such that for each geometric point $s \in S$} \\
\text{the sublattice $\gamma'_s(\Gamma(2)) \hookrightarrow \mathrm{NS}(\CX_s)$ } \\
\text{contains an ample line bundle} \\
\text{and $\gamma'_s(\Gamma(2))^{\perp} \hookrightarrow \mathrm{NS}(\CX_s)$ contains no $(-2)$-vector}\end{array} \right\} \end{align*}
We are again interested in the representability of the functor $\underline{\tilde{\CS}}''_{\gamma}$. Consider the group $O(N_{\sigma},\gamma) = \{\varphi \in O(N_{\sigma}) \mid \varphi \circ \gamma = \gamma\}$ of isometries of $N_{\sigma}$ that preserve the embedding $\gamma$. Since the canonical homomorphism $O(N_{\sigma}) \rightarrow O(\gamma(\Gamma(2)) \oplus \gamma(\Gamma(2))^{\perp})$ is injective, the group $O(N_{\sigma}, \gamma)$ is a subgroup of $O(\gamma(\Gamma(2))^{\perp})$  and the latter group is finite because the lattice $\gamma(\Gamma(2))^{\perp}$ is negative definite. Hence it follows that $O(N_{\sigma}, \gamma)$ is a finite group. There is a group action of $O(N_{\sigma}, \gamma)$ on the functor ${\underline{\tilde{\CS}}''_{\gamma}}^{ps}$ which is given on $S$-valued points for connected schemes $S$ via 
\begin{align*} \varphi \cdot (f \colon \CX \rightarrow S, \psi \colon \underline{N}_{\sigma} \rightarrow \Pic_{\CX/S}) = (f \colon \CX \rightarrow S, \psi \circ \varphi \colon \underline{N}_{\sigma} \rightarrow \Pic_{\CX/S}). \end{align*}
\begin{proposition}\label{quoti1} There is a canonical isomorphism of functors $F \colon \underline{\CS}''_{\gamma} / O(N_{\sigma}, \gamma) \rightarrow {\underline{\tilde{\CS}}''_{\gamma}}^{ps}$. \end{proposition}
\begin{proof} 
There is a canonical morphism of functors $F' \colon \underline{\CS}''_{\gamma} \rightarrow {\underline{\tilde{\CS}}''_{\gamma}}^{ps}$ which is given on $S$-valued points via
\begin{align*} (f \colon \CX \rightarrow S, \psi\colon \underline{N}_{\sigma} \rightarrow \Pic_{\CX/S}) \longmapsto (f \colon \CX \rightarrow S, \psi(\underline{N}_{\sigma}) \subseteq \Pic_{\CX/S}, \psi \circ \gamma \colon \Gamma(2) \hookrightarrow \psi(\underline{N}_{\sigma})).  \end{align*}
This morphism is invariant under the action of $O(N_{\sigma},\gamma)$ on $\underline{\CS}''_{\gamma}$ and therefore it descends to a morphism of functors $F \colon \underline{\CS}''_{\gamma}/O(N_{\sigma},\gamma) \rightarrow {\underline{\tilde{\CS}}''_{\gamma}}^{ps}$. We want to show that $F$ is an isomorphism of functors by checking that for any $\mathbb{F}_p$-scheme $S$ the induced map of sets $F(S)$ is a bijection. 

a) Surjectivity: It suffices to show that the map $F'(S) \colon \underline{\CS}'_{\gamma}(S) \rightarrow {\underline{\tilde{\CS}}''_{\gamma}}^{ps}(S)$ is surjective. To this end, we consider an element 
$s= (f, \underline{\CR}, \gamma') \in {\underline{\tilde{\CS}}''_{\gamma}}^{ps}(S)$ and we choose an isomorphism of lattice embeddings $\psi \colon (\gamma \colon \underline{\Gamma}(2) \hookrightarrow \underline{N}_{\sigma}) \stackrel{\sim}\lra (\gamma' \colon \underline{\Gamma}(2) \hookrightarrow \underline{\CR})$. Then the pair $s'=(f, \psi) \in \underline{\CS}''_{\gamma}(S)$ is a preimage of $s$ under $F'$. 

b) Injectivity: For an element $s =  (f, \underline{\CR}, \gamma') \in {\underline{\tilde{\CS}}''_{\gamma}}^{ps}(S)$ we have to show that any two preimages $s'$ and $s''$ in $\underline{\CS}_{\sigma}(S)$ only differ by some isometry $\varphi \in O(N_{\sigma}, \gamma)$. To this end, we write $s' = (f, \psi')$ and $s''= (f , \psi'')$. We find that $\psi'^{-1}|_{\underline{\CR}} \circ \psi'' \in O(N_{\sigma},\gamma)$ and we obtain the equality $(\psi'^{-1}|_{\underline{\CR}} \circ \psi'') \cdot s' = s''$. 
\end{proof}
The action of $O(N_{\sigma},\gamma)$ on $\underline{\CS}''_{\gamma}$ in general has non-trivial stabilizers, so the functor $\underline{\tilde{\CS}}''_{\gamma}$ is not representable by an algebraic space. We introduce the corresponding stack
\begin{align*} {\gls{St''_g}} \colon \CA_{\mathbb{F}_p}^{\mathrm{op}} &\lra \left(\text{Groupoids} \right) \\
S &\longmapsto \left\{\begin{array}{l}\text{Tuples $(f,\underline{\CR}, \gamma')$, where} \\ \text{$f \colon \CX \rightarrow S$ is a family of supersingular K3 surfaces,} \\
\text{$\underline{\CR} \subseteq \Pic_{\CX/S}$ is a subsheaf of lattices and} \\
\text{$\gamma' \colon \underline{\Gamma}(2) \hookrightarrow \underline{\CR}$ is an embedding such that fppf locally} \\
\text{there exists $\alpha \colon \underline{N}_{\sigma} \cong \underline{\CR}$ with $\alpha \circ \gamma = \gamma'$ and} \\
\text{such that for each geometric point $s \in S$} \\
\text{the sublattice $\gamma'_s(\Gamma(2)) \hookrightarrow \mathrm{NS}(\CX_s)$ } \\
\text{contains an ample line bundle} \\
\text{and $\gamma'_s(\Gamma(2))^{\perp} \hookrightarrow \mathrm{NS}(\CX_s)$ contains no $(-2)$-vector}\end{array} \right\}. \end{align*}

\begin{proposition}\label{quoti}
The canonical morphism $\CS''_{\gamma} \lra \tilde{\mathfrak{S}}''_{\gamma}$ is surjective, finite and étale. In particular $\tilde{\mathfrak{S}}''_{\gamma}$ is a smooth Deligne-Mumford stack.
\end{proposition}
\begin{proof}
Proposition \ref{quoti1} implies that ${\tilde{\mathfrak{S}}''_{\gamma}}$ is isomorphic the quotient stack $[\underline{\CS}''_{\gamma}/O(N_{\sigma},\gamma)]$.
\end{proof}
The action of $O(N_{\sigma}, \gamma)$ on $\underline{\CS}''_{\gamma}$ descents to an action on $\underline{\CM}'_{\gamma}$ in the following way: if $G \in \underline{\CM}'_{\gamma}(S)$ and $\varphi \in O(N_{\sigma}, \gamma)$, then $\varphi \cdot G = \varphi(G) \subseteq N_{\sigma} \otimes \CO_S$. Analogously, there is another description of the action on $\underline{\CS}''_{\gamma}$ via the period isomorphism to $\underline{\tilde{\CP}}''_{\gamma}$: let $(G, \alpha) \in \underline{\tilde{\CP}}''_{\gamma}(S)$, then $\varphi \cdot (G, \alpha) = (\varphi(G), \alpha)$. It follows from Proposition \ref{amplec} that this notation makes sense.

We define stacks
\begin{align*} {\gls{Mft''_g}} \colon \CA_{\mathbb{F}_p}^{\mathrm{op}} &\lra \left(\text{Groupoids} \right) \\
S &\longmapsto \left\{\begin{array}{l} \text{equivalence classes $[G]$ of} \\
\text{characteristic generatrices $G \subseteq pN_{\sigma}^{\vee}/pN_{\sigma} \otimes \CO_S$}\\
\text{under the $O(N_{\sigma}, \gamma)$-action} \\
\text{such that $\gamma(\Gamma(2))^{\perp} \cap \Delta_{N_{\sigma}(s)} = \emptyset$} \end{array} \right\} \intertext{and}
\tilde{\mathfrak{P}}''_{\gamma} \colon \CA_{\mathbb{F}_p}^{\mathrm{op}} &\lra \left(\text{Groupoids} \right) \\
S &\longmapsto \left\{\begin{array}{l}\text{Tuples $([G],\alpha)$, where} \\ 
\text{$[G]$ is an equivalence class of characteristic generatrices} \\
\text{$G \subseteq pN_{\sigma}^{\vee}/pN_{\sigma}  \otimes \CO_S$ under the $O(N_{\sigma}, \gamma)$-action} \\
\text{such that $\gamma(\Gamma(2))^{\perp} \cap \Delta_{N_{\sigma}(s)} = \emptyset$ for all $s \in S$} \\
\text{and $\alpha \in C_{\Gamma(2)}^{|S|}$ is the choice of an ample cone} \end{array} \right\}. \end{align*}

\begin{proposition}
The stacks ${\tilde{\mathfrak{M}}''_{\gamma}}$ and $\tilde{\mathfrak{P}}''_{\gamma}$ are smooth Deligne-Mumford stacks, the canonical morphism ${\tilde{\mathfrak{S}}''_{\gamma}} \lra {\tilde{\mathfrak{M}}''_{\gamma}}$ is an étale surjection and the canonical morphism ${\tilde{\mathfrak{S}}''_{\gamma}} \lra \tilde{\mathfrak{P}}''_{\gamma}$ is an isomorphism.
\end{proposition}
\begin{proof}
This is clear since ${\tilde{\mathfrak{M}}''_{\gamma}}$ and $\tilde{\mathfrak{P}}''_{\gamma}$ are the quotient stacks associated to the action of $O(N_{\sigma}, \gamma)$ on ${\tilde{\CM}''_{\gamma}}$ and ${\tilde{\CP}''_{\gamma}}$ respectively.
\end{proof}
\begin{proposition}
There exists a quasi-projective $\mathbb{F}_p$-scheme $\gls{Mt''_g}$ and a finite $O(N_{\sigma}, \gamma)$-equivariant morphism $q \colon \CM_{\gamma}'  \lra \widetilde{\CM}_{\gamma}''$ such that $q$ is a strongly geometric quotient for the action of $O(N_{\sigma}, \gamma)$ on $\CM_{\gamma}'$ in the sense of \cite[Definition 2.2]{MR3084720}. The pair $(\widetilde{\CM}_{\gamma}'',q)$ is unique upto unique isomorphism.
\end{proposition}
\begin{proof}
When a strongly geometrical quotient exists, it is always unique. The existence follows from \cite[Theorem 4.4]{MR3084720}. The scheme $\widetilde{\CM}''_{\gamma}$ is quasi-projective and $q$ is finite by \cite[Proposition 4.7.]{MR3084720}.
\end{proof}
We will use the scheme $\widetilde{\CM}''_{\gamma}$ later to construct the period scheme of supersingular Enriques surfaces.
\begin{remark}\label{struct}
The scheme $\CP''_{\gamma}$ can be described in the following way.  Choose a totally isotropic subspace $\Lambda \subseteq pN_{\sigma}^{\vee}/pN_{\sigma}$. Let $G_{\CM'_{\gamma}}^{\Lambda}$ be the open subset of $G_{\CM'_{\gamma}}$ corresponding to generatrices $G$ with $G \cap pN_{\sigma}^{\vee}/pN_{\sigma} \subseteq \Lambda$. For each $\alpha$ of the universal generatrix on $G_{\CM'_{\gamma}}^{\Lambda}$ let $G_{\CM'_{\gamma}}^{\Lambda, \alpha}$ be a copy of $G_{\CM'_{\gamma}}^{\Lambda}$. For $\Lambda' \subseteq \Lambda$ one has $G_{\CM'_{\gamma}}^{\Lambda'} \subseteq G_{\CM'_{\gamma}}^{\Lambda}$. Then $\CP''_{\gamma}$ is the scheme obtained from the disjoint union of the $G_{\CM'_{\gamma}}^{\Lambda, \alpha}$ by identifying open $G_{\CM'_{\gamma}}^{\Lambda', \alpha'}$ with the open subsets $G_{\CM'_{\gamma}}^{\Lambda', \alpha} \subseteq G_{\CM'_{\gamma}}^{\Lambda, \alpha}$ if $\alpha$ and $\alpha'$ agree there. For details, we refer to the proof of \cite[Proposition 1.16]{MR717616}.

In particular, there is an open cover $\CP''_{\gamma} = \bigcup_{\Lambda, \alpha} {\CM'_{\gamma}}^{\Lambda, \alpha}$ by open subsets of $G_{\CM'_{\gamma}}$. For a fixed ample cone $\alpha$ the union ${\CM'_{\gamma}}^{\alpha} \coloneqq \bigcup_{\Lambda} {\CM'_{\gamma}}^{\Lambda, \alpha}$ is an open subscheme of ${\CM'_{\gamma}}$ and each of the ${\CM'_{\gamma}}^{\alpha}$ is closed under the action of $O(N_{\sigma}, \gamma)$ on $\CP''_{\gamma}$.

Taking the corresponding quotient stacks by the $O(N_{\sigma}, \gamma)$-action, we also get a cover of $\tilde{\mathfrak{P}}''_{\gamma}$ via open immersions ${{\tilde{\mathfrak{M}}}_{\gamma}}^{\prime \prime \alpha} \lra \tilde{\mathfrak{P}}''_{\gamma}$.
\end{remark}

\begin{proposition}
There exists an $\mathbb{F}_p$-scheme $\gls{Stc''_s}$ locally of finite type and a $O(N_{\sigma}, \gamma)$-equivariant morphism $q \colon \CS_{\gamma}'' \lra \tilde{\CS}_{\gamma}''$ such that $q$ is a strongly geometric quotient for the action of $O(N_{\sigma}, \gamma)$ on $\CS_{\gamma}''$ in the sense of \cite[Definition 2.2]{MR3084720}. The pair $(\tilde{\CS}_{\gamma}'',q)$ is unique upto unique isomorphism and the induced morphism $\tilde{\pi}''_{\gamma} \colon \tilde{\CS}_{\gamma}'' \lra \widetilde{\CM}''_{\gamma}$ is étale and surjective.
\end{proposition}
\begin{proof}
It follows from Remark \ref{struct} that we can apply \cite[Theorem 4.4]{MR3084720} and \cite[Proposition 4.7.]{MR3084720} again. It remains to show that $\tilde{\pi}''_{\gamma} \colon \tilde{\CS}_{\gamma}'' \lra \widetilde{\CM}''_{\gamma}$ is étale.

We claim that the morphism $\pi_{\gamma}''$ is \emph{fixed-point reflecting} in the sense of \cite{MR3084720}. That means for each $x \in \CS''_{\gamma}$ and $\varphi \in O(N_{\sigma},\gamma)$ we have that $\varphi \cdot x = x$ if and only if $\varphi \cdot \pi''_{\gamma}(x)= \pi''_{\gamma}(x)$. Indeed, let $x \in \CS''_{\gamma}(k)$ and $\varphi \in O(N_{\sigma},\gamma)$ such that $x$ corresponds to a tuple $(G, \alpha) \in \CP_{\sigma}(k)$ where $G$ is a characteristic subspace and $\alpha$ is an ample cone. We need to verify that if $\varphi \cdot G = G$, then we also have $\varphi \cdot (G, \alpha) = (G, \alpha)$. But this is clear from the previous discussion. 

Since the quotient $\CM'_{\gamma} \rightarrow \widetilde{\CM}''_{\gamma}$ satisfies the \textit{descent condition} in the sense of \cite[Definition 3.6]{MR3084720} we are done.
\end{proof}

\begin{definition}
We write $\underline{\CS}''_{\gamma, \sigma}$ for the open subfunctor of $\underline{\CS}''_{\gamma}$ such that
\begin{align*}
    \underline{\CS}''_{\gamma, \sigma}(S) = \{(f \colon \CX \rightarrow S,\psi) \in \underline{\CS}''_{\gamma}(s) \mid \sigma(\mathrm{NS}(\CX_s)) = \sigma \text{ for all $s \in S$}\}
\end{align*}
and define $\CS''_{\gamma, \sigma}$, $\underline{\tilde{\CS}}''_{\gamma, \sigma}$, $\CM''_{\gamma, \sigma}$ etc.\ analogously.
\end{definition}

\begin{proposition}\label{coar}
The open subscheme $q(\CS''_{\gamma, \sigma})$ in $\tilde{\CS}''_{\gamma}$ represents  $\underline{\tilde{\CS}}''_{\gamma, \sigma}$. 
\end{proposition}
\begin{proof}
The open subscheme $\CS''_{\gamma, \sigma}$ is closed under the action of $O(N_{\sigma}, \gamma)$ on $\CS''_{\gamma}$. By \cite[Theorem 2.16]{MR3084720} it suffices to show that the induced action of $O(N_{\sigma}, \gamma)$ on $\CS''_{\gamma, \sigma}$ factorizes over a free group action. 

Let $G \in \underline{\CM}''_{\gamma, \sigma}(k)$. Using the description of the action of $O(N_{\sigma}, \gamma)$ on $\underline{\CS}''_{\gamma}$ via the period map, it suffices to show that if $\varphi \in O(N_{\sigma}, \gamma)$ is such that $\varphi(G) = G$, then we have $\varphi(G')=G'$ for all $G' \in \underline{\CM}''_{\gamma, \sigma}(S)$.

A generatrix $G \in \underline{\CM}''_{\gamma, \sigma}(k)$ is always a strictly characteristic generatrix of $pN_{\sigma}^{\vee}/pN_{\sigma} \otimes k$ and writing $F$ for the induced Frobenius morphism on $pN_{\sigma}^{\vee}/pN_{\sigma} \otimes k$ we can write $G=\{x_0, \ldots, x_{\sigma-1}\}$ with $\langle x_0 \rangle = \bigcap_{i=0}^{\sigma-1} F^{-i}(G)$ and $F(x_i) = x_{i+1}$ for all $i$. Then $F^{\sigma}(G)$ is also a strictly characteristic generatrix and we have $(F(x_i), x_j) \neq 0$ if and only if $i=j$. In particular it follows that $\{x_0, F(x_0), \ldots, F^{2 \sigma-1}(x_0)\}$ is a basis of $pN_{\sigma}^{\vee}/pN_{\sigma} \otimes k$.

We observe that for any $\varphi \in O(N_{\sigma}, \gamma)$ the induced automorphism $\varphi$ on $pN_{\sigma}^{\vee}/pN_{\sigma} \otimes k$ commutes with $F$ (since it comes from an automorphism of $pN_{\sigma}^{\vee}/pN_{\sigma}$). Let $\varphi \in O(N_{\sigma}, \gamma)$ with $\varphi(G) = G$. Then $x_0$ is an eigenvector of $\varphi$, say $\varphi(x_0) = a \cdot x_0$ and subsequently we find $\varphi(F^i(x_0)) = a \cdot F^i(x_0)$ for all $i$. Thus, $\varphi \colon pN_{\sigma}^{\vee}/pN_{\sigma} \otimes k \lra pN_{\sigma}^{\vee}/pN_{\sigma} \otimes k$ is just multiplication by $a$ (a postiori we even have $a \in \{\pm 1\}$) and thus fixes any $G' \in \underline{\CM}''_{\gamma, \sigma}(k)$. 
\end{proof}




\section{Moduli spaces of \texorpdfstring{$\Gamma(2)$-marked}{marked} supersingular K3 surfaces}\label{secmark}
Next, we want to get rid of having to make a choice of a sublattice $\underline{\CR}$ in $\Pic_{\CX/S}$. The idea is, that on an open dense subset of the moduli stack $\tilde{\mathfrak{S}}''_{\gamma}$ we do not have a choice anyways, and the closed complement of this open subspace can be contracted to the corresponding moduli stack for Artin invariant $\sigma-1$ by forgetting about the sublattice $\underline{\CR}$.

We now introduce the stack
\begin{align*} \gls{Et_s} \colon \CA_{\mathbb{F}_p}^{\mathrm{op}} &\lra \left(\text{Groupoids} \right) \\
S &\longmapsto \left\{\begin{array}{l}\text{Tuples $(f, \gamma)$, where} \\ \text{$f \colon \CX \rightarrow S$ is a family of supersingular K3 surfaces} \\
\text{that fppf locally admit an $N_{\sigma}$-marking} \\
\text{and $\gamma$ is an embedding $\gamma \colon \underline{\Gamma}(2) \hookrightarrow \Pic_{\CX/S}$} \\
\text{such that for each geometric point $s \in S$} \\
\text{the sublattice $\gamma_s(\Gamma(2)) \hookrightarrow \mathrm{NS}(\CX_s)$} \\
\text{contains an ample line bundle} \\
\text{and $\gamma_s(\Gamma(2))^{\perp} \hookrightarrow \mathrm{NS}(\CX_s)$ contains no $-2$-vector}\end{array} \right\}.  \end{align*}
We are again interested in the properties of the stack $\tilde{\mathfrak{E}}_{\sigma}$. The discussion will use an inductive argument, so we start by discussing the case $\sigma=1$.
\begin{proposition}\label{sigmaisone} The stack $\tilde{\mathfrak{E}}_{1}$ is representable by a zero-dimensional scheme $\tilde{\mathfrak{E}}_1$ locally of finite type over $\mathbb{F}_p$. \end{proposition}
\begin{proof} For each $\gamma \in R_1$ there is a canonical morphism of functors $\underline{\tilde{\CS}}''_{\gamma} \rightarrow \tilde{\mathfrak{E}}_1$ which is given on $S$-valued points by forgetting about the choice of a sublattice $\underline{\CR} \subseteq \Pic_{\CX/S}$. Since any such sublattice $\underline{\CR} \subseteq \Pic_{\CX/S}$ is actually already equal to $\Pic_{\CX/S}$, we see that this morphism is injective on $S$-valued points and it follows that $\coprod_{\gamma \in R_1} \underline{\tilde{\CS}}''_{\gamma} \lra \tilde{\mathfrak{E}}_1$ is an isomorphism of functors. Hence, the functor $\tilde{\mathfrak{E}}_1$ is represented by the scheme $\coprod_{\gamma \in R_1} \tilde{\CS}''_{\gamma}$. \end{proof}
\begin{remark} More precisely, since $\CS_{1}$ is isomorphic to a disjoint union of finitely many copies of $\Spec \mathbb{F}_{p^2}$ and $\tilde{\mathfrak{S}}''_{\gamma}$ is just an open subscheme of a quotient of an open subscheme of $\CS_{\sigma}$, we easily see that $\tilde{\mathfrak{E}}_{1}$ is just a disjoint union of finitely many copies of $\Spec \mathbb{F}_{p^2}$ as well. In particular $\tilde{\mathfrak{E}}_1$ is a separated $\mathbb{F}_p$-scheme. However, for the inductive argument later we will only need the properties from Proposition \ref{sigmaisone}. \end{remark}
We will need the following lemma which might be well-known, but we did not find it in the literature in full generality. That is, we do not require any assumptions on being a scheme, being noetherian or separatedness.
\begin{lemma}\label{sweet} Let $X, Y$ and $Z$ be algebraic spaces that are locally of finite type over a base scheme $S$ together with $S$-morphisms $f \colon X \rightarrow Y$ and $g \colon Y \rightarrow Z$ such that $g \circ f$ is proper (respectively finite) and $f$ is proper (respectively finite) and surjective. Then $g$ is proper (respectively finite). \end{lemma}
\begin{proof} We prove that $g$ is finite when $f$ and $g \circ f$ are finite. We leave the proper case to the reader. Since $Y$ and $Z$ are locally of finite type, the morphism $g$ is locally of finite type \cite[Lemma 61.23.6]{stacks-project}. It is clear that $g$ has finite discrete fibers, because the fibers of $g \circ f$ surject onto the fibers of $g$. Further, the morphism $g$ is quasi-compact \cite[Lemma 61.8.6]{stacks-project}. It follows that $g$ is quasi-finite. Further, if $T \rightarrow Z$ is any morphism and $Q \subseteq Y_T$ is a closed subscheme, then the subscheme $g_T(Q)=g_T \circ f_T(f^{-1}_T(Q))$ is closed. This shows that $g$ is universally closed. Further, the fact that $g$ is affine follows from a version of Chevalley's theorem \cite[Theorem 8.1]{MR3272071}. All these properties together imply that $g$ is finite. \end{proof}
Since every family of supersingular K3 surfaces that admits an $N_{\sigma-1}$-marking also admits an $N_{\sigma}$-marking, the stack $\tilde{\mathfrak{E}}_{\sigma-1}$ is a substack of $\tilde{\mathfrak{E}}_{\sigma}$. For each positive integer $\sigma \leq 10$ there is a canonical morphism of stacks \begin{align*} p_{\sigma}: \coprod_{\gamma \in R_{\sigma}} \underline{\tilde{\mathfrak{S}}}''_{\gamma} \rightarrow \tilde{\mathfrak{E}}_{\sigma} \end{align*}
which is given on $S$-valued points by forgetting about the sublattice $\underline{\CR} \subseteq \Pic_{\CX/S}$. Then the preimage of the substack $\tilde{\mathfrak{E}}_{\sigma-1} \hookrightarrow \tilde{\mathfrak{E}}_{\sigma}$ under $p_{\sigma}$ is given by the closed substack \begin{align*} p_{\sigma}^{-1}(\tilde{\mathfrak{E}}_{\sigma-1}) = \coprod_{\gamma \in R_{\sigma}} \left(\left(\bigcup_{j \in R_{\sigma-1,\sigma}} q\left(\Phi_j(\CS_{\sigma-1})\cap \CS''_{\gamma} \right) \right) \backslash \left( \bigcup_{j \in R_{\gamma}} q\left(\Phi_j(\CS_{\sigma-1}) \cap \CS''_{\gamma} \right)\right) \right) \end{align*}  
of $\coprod_{\gamma \in R_{\sigma}} \tilde{\mathfrak{S}}''_{\gamma}$.
\begin{definition} For $\gamma \in R_{\sigma}$ and $j \in R_{\sigma-1, \sigma}$, we write $W^{\gamma}_{j}$ for the locally closed subspace of $\CS_{\sigma}$ defined to be
$$W^{\gamma}_{j} =  \left( \Phi_j(\CS_{\sigma-1}) \cap \CS''_{\gamma} \right) \backslash \left(\bigcup_{j' \in R_{\gamma}} \Phi_j(\CS_{\sigma-1}) \cap \CS''_{\gamma}\right).$$  \end{definition}
\begin{remark}\label{w} We have an equality of closed substacks $\bigcup_{\gamma \in R_{\sigma}, j \in R_{\sigma-1, \sigma}} q(W^{\gamma}_{j}) = p_{\sigma}^{-1}(\tilde{\mathfrak{E}}_{\sigma-1})$. Moreover, since $W^{\gamma}_{j}$ is a closed subspace of $\CS''_{\gamma}$, it follows from Proposition \ref{quoti} that the morphism $q|_{W^{\gamma}_{j}} \colon W^{\gamma}_{j} \lra \tilde{\mathfrak{S}}''_{\gamma}$ is finite.

Further, since $\Phi_j(\CS_{\sigma-1}) \cap \CS''_{\gamma}$ is canonically isomorphic to the open subscheme $\CS''_{j \circ \gamma}$ of $\CS_{\sigma-1}$, we also have a natural finite morphism $q \colon W^{\gamma}_j \lra \tilde{\mathfrak{S}}''_{j \circ \gamma}$. \end{remark}
\begin{lemma}\label{finite} Assume that $\tilde{\mathfrak{E}}_{\sigma-1}$ is a Deligne-Mumford stack that is locally of finite type over $\mathbb{F}_p$ and that the canonical morphism \mbox{$\coprod_{\gamma \in R_{\sigma}, j \in R_{\sigma-1, \sigma} } W^{\gamma}_{j} \rightarrow \tilde{\mathfrak{E}}_{\sigma-1}$} is finite. Then the restriction of $p_{\sigma}$ to $p_{\sigma}^{-1}(\tilde{\mathfrak{E}}_{\sigma-1})$ is a finite morphism . \end{lemma}
\begin{proof} This is a direct consequence of Lemma \ref{sweet}, and the previous remark. \end{proof}
\begin{theorem}\label{theorem1} Let $\sigma \leq 10$ be a positive integer.
\begin{enumerate}
\item The stack $\tilde{\mathfrak{E}}_{\sigma}$ is a Deligne-Mumford stack which is locally of finite type over $\mathbb{F}_p$ and quasi-separated.
\item  For each isomorphism class of primitive embeddings $\gamma \colon \Gamma(2) \hookrightarrow N_{\sigma+1}$ such that there is no $(-2)$-vector in $\gamma(\Gamma(2))^{\perp} \subset N_{\sigma+1}$ and each embedding of lattices $j \colon N_{\sigma+1} \hookrightarrow N_{\sigma}$ such that there is no $(-2)$-vector in $j(\gamma(\Gamma(2)))^{\perp} \subset N_{\sigma}$, the canonical morphism $W^{\gamma}_{j} \rightarrow \tilde{\mathfrak{E}}_{\sigma}$ is finite.
\end{enumerate}  \end{theorem}
\begin{proof}
 We do induction over $\sigma$. For $\sigma=1$, the theorem follows from Proposition \ref{sigmaisone} and its proof.
 
 We will now assume that the theorem holds for all integers lesser or equal then $\sigma-1$. By Lemma \ref{finite} the morphism $p_{\sigma} \colon p^{-1}_{\sigma}(\tilde{\mathfrak{E}}_{\sigma-1}) \rightarrow \tilde{\mathfrak{E}}_{\sigma-1}$ is finite, hence by \cite[Theorem A.4]{ryd2011}
 there exists a pushout $\mathfrak{B}_{\sigma}$ in the category of algebraic stacks fitting into the cartesian diagram
\begin{align*} \xymatrix{p^{-1}_{\sigma}(\tilde{\mathfrak{E}}_{\sigma-1}) \ar[r]_{\iota} \ar[d]_{p_{\sigma}}\po & \coprod_{\gamma \in R_{\sigma}} \tilde{\mathfrak{S}}''_{\gamma}\ar[d]\\
    \tilde{\mathfrak{E}}_{\sigma-1} \ar[r]& \mathfrak{B}_{\sigma}.}\end{align*}
The pushout $\mathfrak{B}_{\sigma}$ is a Deligne-Mumford stack locally of finite-type over $\mathbb{F}_p$. Further, the morphism $\coprod_{\gamma \in R_{\sigma}} \tilde{\mathfrak{S}}''_{\gamma} \lra \mathfrak{B}_{\sigma}$ is surjective (and we will see that it is finite) and the morphism $\tilde{\mathfrak{E}}_{\sigma-1} \lra \mathfrak{B}_{\sigma}$ is a closed immersion.

The topological space attached to $\mathfrak{B}_{\sigma}$ is just the pushout in the category of topological spaces, there exists a natural isomorphism of DM-stacks $p^{-1}_{\sigma}(\tilde{\mathfrak{E}}_{\sigma-1}) \cong \tilde{\mathfrak{E}}_{\sigma-1} \times_{\CP} \coprod_{\gamma \in R_{\sigma}} \tilde{\CS}''_{\gamma}$, the morphism $\left(\coprod_{\gamma \in R_{\sigma}} \tilde{\mathfrak{S}}''_{\gamma} \right) \backslash \left(p^{-1}_{\sigma}(\tilde{\mathfrak{E}}_{\sigma-1})\right) = U \rightarrow \CP$ is an open immersion of DM-stacks and we have an equality of sets $\left|\mathfrak{B}_{\sigma} \right| = \left|\tilde{\mathfrak{E}}_{\sigma-1}\right| \amalg \left|U\right|$. 

Similary, we construct three more auxiliary pushouts $\mathfrak{B}_{atl}^{\sigma}$, $\mathfrak{B}_{sep}^{\sigma}$ and $\mathfrak{B}_{uni}^{\sigma}$. The functors of these pushouts might be interesting in their own right, but we do not want to over complicate things and therefore abstain from describing them.

Each of those $\mathfrak{B}_i^{\sigma}$ for $i \in \{atl, cov, uni\}$ is constructed inductively in the following way: for $\sigma= 1$ we start with a DM-stack $\mathfrak{B}_i^1$, a morphism $\mathfrak{B}_i^1 \lra \tilde{\mathfrak{E}}_1$, some DM-stack $\mathfrak{A}_i^2$ and a morphism $\mathfrak{A}_i^2 \lra \coprod_{\gamma \in R_{\sigma}} \tilde{\mathfrak{S}}''_{\gamma}$ such that the pullback of $p_{\sigma} \colon p^{-1}_{\sigma}(\tilde{\mathfrak{E}}_1) \lra \tilde{\mathfrak{E}}_1$ to $\mathfrak{B}_i^1$ is a finite surjective morphism $p_{\sigma} \colon p^{-1}_{\sigma}(\mathfrak{B}_i^1) \lra \mathfrak{B}_i^1$ and the pullback of $\iota \colon p^{-1}_{\sigma}(\tilde{\mathfrak{E}}_1) \hookrightarrow \coprod_{\gamma \in R_{\sigma}} \tilde{\mathfrak{S}}''_{\gamma}$ is a closed immersion $\iota  \colon p^{-1}_{\sigma}(\mathfrak{B}_i^1) \hookrightarrow \mathfrak{A}_i^2$. We then take $\mathfrak{B}_i^2$ to be the pushout in the category of DM-stacks of the pushout datum $\mathfrak{B}_i^1 \leftarrow  p^{-1}_{\sigma}(\mathfrak{B}_i^1) \rightarrow \mathfrak{A}_i^2$. This induces a morphism of DM-stacks $\mathfrak{B}_i^2 \lra \mathfrak{B}_2$. Inductively, for $l \leq \sigma -1$, we may assume that $\mathfrak{B}_l \cong \tilde{\mathfrak{E}}_l$ and by choosing appropriate $\mathfrak{A}_i^l$ we can construct $B_i^{1}, \ldots, B_i^{\sigma}$.

We define the $\mathfrak{B}_i^l$ in the following way:
\begin{itemize}
    \item We set $\mathfrak{B}_{atl}^1 = \coprod_{\gamma \in R_{1}} \CS''_{\gamma}$ and $\mathfrak{A}_{atl}^l  = \coprod_{\gamma \in R_{l+1}} \CS''_{\gamma}$.
    \item Following Remark \ref{struct} we set 
    \begin{align*} \mathfrak{B}_{cov}^1 &= \coprod_{\gamma \in R_{1}} \left( \coprod_{\Lambda, \alpha} {\CM'_{\gamma}}^{\Lambda, \alpha}\right)
    \intertext{and} 
    \mathfrak{A}_{cov}^l& =  \coprod_{\gamma \in R_{l+1}} \left( \coprod_{\Lambda, \alpha} {\CM'_{\gamma}}^{\Lambda, \alpha}\right). 
    \intertext{We also write}
    \mathfrak{B}_{cov}^l &= \coprod_{\gamma \in R_{l+1}} \mathfrak{B}_{cov}^{l, \Lambda, \alpha}
    \end{align*}
    \item Writing $\CX_{\gamma}$ for the universal family over $\CS_{\gamma}''$ we set $\mathfrak{B}_{uni}^1 = \coprod_{\gamma \in R_{1}} \CX_{\gamma}$ and $\mathfrak{A}_{uni}^l  = \coprod_{\gamma \in R_{l+1}} \CX_{\gamma}$.
\end{itemize}
It is clear that for these choices of $\mathfrak{B}_i^l$ and $\mathfrak{A}_i^l$ the conditions in the construction are okay. The morphisms in the construction of $\mathfrak{B}_{cov}^1$ are compatible with the open affine coverings $\mathfrak{B}_{cov}^l \lra  \coprod_{\gamma \in R_{l}} \CM'_{\gamma}$. We therefore also get pushouts $\mathfrak{B}_{gen}^1$ with
\begin{itemize}
    \item $\mathfrak{B}_{gen}^1 = \coprod_{\gamma \in R_{1}} \CM'_{\gamma}$ and $\mathfrak{A}_{gen}^l = \coprod_{\gamma \in R_{l+1}} \CM'_{\gamma}$.
\end{itemize}

The $\mathfrak{B}_{cov}^l$ and $\mathfrak{B}_{gen}^l$ are separated AF-schemes locally of finite type over $\mathbb{F}_p$ by \cite[Theorem A.4]{ryd2011} and the $\mathfrak{B}_{uni}^l$ and $\mathfrak{B}_{atl}^l$ are quasi-separated algebraic spaces locally of finite type over $\mathbb{F}_p$ by \cite[Theorem A.4]{ryd2011}. 

We obtain morphisms
\begin{align*}
    \mathfrak{B}_{atl}^l &\lra \mathfrak{B}_{gen}^l, \\
    \mathfrak{B}_{cov}^l &\lra \mathfrak{B}_{atl}^l  \\
    \mathfrak{B}_{uni}^l &\lra \mathfrak{B}_{atl}^l  \\
        \intertext{and}
    \mathfrak{B}_{atl}^l &\lra \mathfrak{B}_l.
\end{align*}
induced from flat morphisms of pushout diagrams. Inductively, since these properties are preserved by pushouts, each of these morphisms is smooth and surjective.

The morphisms of algebraic spaces $\mathfrak{A}_i^l \lra \mathfrak{B}_{atl}^l$ are finite \cite[Theorem 6.6]{MR3572553} and since being finite is local in the fppf-topology it follows that $\coprod_{\gamma \in R_{\sigma}} \tilde{\mathfrak{S}}''_{\gamma} \lra \mathfrak{B}_{\sigma}$ is finite. 

\claim  $\mathfrak{B}_{uni}^l \lra \mathfrak{B}_{atl}^l$ is a family of supersingular K3 surfaces.
\prfclaim Using Lemma \ref{sweet} we also find inductively that $\mathfrak{B}_{uni}^l \lra \mathfrak{B}_{atl}^l$ is proper. Also inductively, since a geometric fiber of $\mathfrak{B}_{uni}^l \lra \mathfrak{B}_{atl}^l$ is a geometric fiber of $\mathfrak{B}_{uni}^{l-1} \lra \mathfrak{B}_{atl}^{l-1}$ or of $\mathfrak{A}_{uni}^{l} \lra \mathfrak{A}_{atl}^{l}$, any geometric fiber of $\mathfrak{B}_{uni}^l \lra \mathfrak{B}_{atl}^l$ is a supersingular K3 surface. Thus $\mathfrak{B}_{uni}^l \lra \mathfrak{B}_{atl}^l$ is a family of K3 surfaces and $\mathfrak{B}_{uni}^l$ is equipped with a canonical $\Gamma(2)$-marking and by the discussion in \cite[page 1522]{MR633161} admits an $N_{l}$-marking étale locally. 

\claim $\mathfrak{B}_{atl}^l$ is a scheme.
\prfclaim It suffices to show that $\mathfrak{B}_{cov}^{l, \Lambda, \alpha} \lra \mathfrak{B}_{atl}^l$ is an open immersion for any $\Lambda$ and $\alpha$. Using  \cite[Theorem 6.4]{MR3572553} the morphisms $\mathfrak{B}_{cov}^{l, \Lambda, \alpha} \lra  \mathfrak{B}_{gen}^l$ are open immersions of separated AF-schemes by induction. Then $\mathfrak{B}_{cov}^{l, \Lambda, \alpha} \lra \mathfrak{B}_{atl}^l$ is a monomorphism of finite presentation \cite[Lemma 65.12.2(1)]{stacks-project} and hence an open immersion \cite[Theorem 41.14.1]{stacks-project}.

We now show that the DM-stack $\mathfrak{B}_{\sigma}$ is isomorphic to the stack $\tilde{\mathfrak{E}}_{\sigma}$ and that the morphism of DM-stacks $\coprod_{\gamma \in R_{\sigma}} \tilde{\mathfrak{S}}''_{\gamma} \rightarrow \mathfrak{B}_{\sigma}$ corresponds to the canonical morphism $\coprod_{\gamma \in R_{\sigma}} \tilde{\mathfrak{S}}''_{\gamma} \rightarrow \tilde{\mathfrak{E}}_{\sigma}$.

\emph{Step 1: We define a morphism of stacks $F \colon \tilde{\mathfrak{E}}_{\sigma} \rightarrow \mathfrak{B}_{\sigma}$.}

The morphism we define will be compatible with taking subschemes (not necessarily closed or open). Hence, it suffices to define it on irreducible schemes, since for a reducible scheme $S$ with irreducible components $S_i$ we can take $F_S \colon \tilde{\mathfrak{E}}_{\sigma}(S) \rightarrow \mathfrak{B}_{\sigma}(S)$ to be the colimit of the $F_{S_i} \colon \tilde{\mathfrak{E}}_{\sigma}(S_i) \rightarrow \mathfrak{B}_{\sigma}(S_i)$.

If $S$ is an irreducible $\mathbb{F}_p$-scheme, we define the functor $F(S) \colon \tilde{\mathfrak{E}}_{\sigma}(S) \rightarrow \mathfrak{B}_{\sigma}(S)$ in the following way. If $x = (f \colon \CX \rightarrow S, \gamma \colon \Gamma(2) \hookrightarrow \Pic_{\CX/S}) \in \tilde{\mathfrak{E}}_{\sigma}(S)$ is such that for every $s \in S$ the fiber $\CX_s$ has Artin invariant $\sigma(\mathrm{NS}(\CX_s)) \leq \sigma-1$, then $x$ is an object of the full subcategory $\tilde{\mathfrak{E}}_{\sigma - 1}(S) \subset \tilde{\mathfrak{E}}_{\sigma}(S)$. In this case, we set $F(S)(x)$ to be the image of $x$ under the canonical functor $\tilde{\mathfrak{E}}_{\sigma - 1}(S) \rightarrow \mathfrak{B}_{\sigma}(S)$. Note, that by the commutativity of the pushout diagram, if $x$ lies in the image of $p_{\sigma}$, we equivalently could have chosen a preimage $x'$ of $x$ in $\tilde{\mathfrak{S}}''_{\gamma}(S)$ for some $\gamma'$ and set $F(S)(x)$ to be the image of $x'$ under the canonical functor $\tilde{\mathfrak{S}}''_{\gamma}(S) \rightarrow \mathfrak{B}_{\sigma}(S)$. 

If, on the other hand, $x$ is such that there exists an $s \in S$ with $\sigma(\mathrm{NS}(\CX_s))=\sigma$, then the subset $U \subseteq S$ where $\CX_s$ has Artin invariant $\sigma$ is open. We first assume that $S$ is reduced and fppf locally choose an arbitrary lift $x'=(f, \CR', \gamma')$ of $x$ to $\tilde{\mathfrak{S}}''_{\gamma}(S)$. We claim that this lift is unique. Indeed, let $x''=(f, \CR'', \gamma'')$ be another such lift. We take preimages $\tilde{x}'=(f, \psi')$ and $\tilde{x}''=(f, \psi'')$ in $\CS'_{\gamma}(S)$ and after applying an automorphism of $N_{\sigma}$ that preserves the embedding $\gamma \colon \Gamma(2) \hookrightarrow N_{\sigma}$, we may assume that $\psi'_U = \psi''_U$. But by \cite[Theorem 3.1.1.]{MR2263236} the morphism of algebraic spaces $\Pic_{\CX/S} \rightarrow S$ is separated and it therefore follows that $\psi'=\psi''$. Thus, we have an isomorphism $x' \cong x''$. Fppf locally we set $F(S)(x)$ to be the image of $x' \in \tilde{\mathfrak{S}}''_{\gamma}(S) $ under the canonical functor $\tilde{\mathfrak{S}}''_{\gamma}(S)  \rightarrow \mathfrak{B}_{\sigma}(S)$ and then glue these images to obtain an object in $\mathfrak{B}_{\sigma}(S)$.

If $S$ is not reduced, let $r \colon S_{\mathrm{red}} \rightarrow S$ be its reduction. Then we have a natural isomorphism $\Pic_{\CX_{\mathrm{red}}/S_{\mathrm{red}}} \cong r^{\ast}\Pic_{\CX/S}$. Since $r$ is a universal homeomorphism, the functors $r^{\ast} \colon Sh(S_{\text{étale}}) \rightarrow Sh({S_{\mathrm{red}}}_{\text{étale}})$ and $r_{\ast} \colon Sh({S_{\mathrm{red}}}_{\text{étale}}) \rightarrow Sh(S_{\text{étale}})$ are mutually quasi-inverse to each other \cite[Proposition 58.45.4]{stacks-project}. In particular, there is a canonical isomorphism $r_{\ast}\Pic_{\CX_{\mathrm{red}}/S_{\mathrm{red}}} \cong \Pic_{\CX/S}$. Let $y'=(r^{\ast}f, \CR, r^{\ast} \gamma)$ be the unique lift of $y=r^{\ast}x$ to $\tilde{\mathfrak{S}}''_{\gamma}(S_{\mathrm{red}})$. Then $x' = (f, r_{\ast}\CR, \gamma)$ is an object of $\tilde{\mathfrak{S}}''_{\gamma}(S)$ and we can take $F(S)(x)$ to be the image of $x' \in \tilde{\mathfrak{S}}''_{\gamma}(S)$ under the canonical functor $\tilde{\mathfrak{S}}''_{\gamma}(S) \rightarrow \CP(S)$

It is clear how the construction extends to morphisms in $\tilde{\mathfrak{E}}_{\sigma}(S)$ and that the class of functors $F(S)$ yields a morphism of stacks.

\emph{Step 2: The functors $F(S) \colon \tilde{\mathfrak{E}}_{\sigma}(S) \rightarrow \mathfrak{B}_{\sigma}(S)$ are (essentially) surjective.} 

Let $x \colon S \lra  \mathfrak{B}_{\sigma}$. Pullback of $x$ to $\mathfrak{B}_{uni}^{\sigma} \lra \mathfrak{B}_{\sigma}$ yields an object $(f_x, \gamma)$ in $\tilde{\mathfrak{E}}_{\sigma}(S)$ and going through the definition in step 1 shows that $F(S)(f_x, \gamma) = x$. An isomorphism $g \colon x \stackrel{\simeq}\lra x'$ corresponds to different choices of pullback to $\mathfrak{B}^{\sigma}_{uni}$ and the induced isomorphism $g^{\ast}$ between pullbacks is an isomorphism of objects in $\tilde{\mathfrak{E}}_{\sigma}(S)$. Again going through the construction, one finds $F(S)(g^{\ast}) = f$.

\emph{Step 3: The DM-stack $\mathfrak{B}_{\sigma}$ is reduced.}

Since for every $S$ any point in $\tilde{\mathfrak{E}}_{\sigma}(S)$ lifts fppf locally to a point in $\coprod_{\gamma \in R_{\sigma}} \tilde{\mathfrak{S}}''_{\gamma}(S)$ and $\tilde{\mathfrak{E}}_{\sigma}(S) \lra \mathfrak{B}_{\sigma}(S)$ is surjective, we obtain fppf locally a factorization 
\begin{equation*} \operatorname{id}_{\mathfrak{B}_{\sigma}} \colon \mathfrak{B}_{\sigma} \lra \coprod_{\gamma \in R_{\sigma}} \tilde{\mathfrak{S}}''_{\gamma} \lra \mathfrak{B}_{\sigma}.
\end{equation*} 
Being reduced is local in the fppf topology, so for the rest of the argument we will assume this split exists globally. Since $\coprod_{\gamma \in R_{\sigma}} \tilde{\mathfrak{S}}''_{\gamma} \lra \mathfrak{B}_{\sigma}$ is finite it follows that $\mathfrak{B}_{\sigma} \lra \coprod_{\gamma \in R_{\sigma}} \tilde{\mathfrak{S}}''_{\gamma}$ is a closed immersion and we are done.

\emph{Step 4: The functors $F(S) \colon \tilde{\mathfrak{E}}_{\sigma}(S) \rightarrow \mathfrak{B}_{\sigma}(S)$ are faithful.}

We may assume $S$ to be reduced (by the previous step) and irreducible. Assume we are given $x=(f \colon \CX \rightarrow S, \gamma)$ and an isomorphism $g \colon x \lra x$ with $F(S)(g) = \mathrm{id}_{F(S)(x)}$. If $x \in \mathfrak{E}_{\sigma-1}(S)$, then we are done. Else, we take fppf local lifts to $\coprod_{\gamma \in R_{\sigma}} \CS''_{\gamma}(S)$ and write $U$ for the open subset of $S$ where $\mathrm{NS}(\CX_s) = \sigma$. We write $Z = S \backslash U$. Being the identity morphism in $\tilde{\mathfrak{E}}_{\sigma}(S)$ is fppf local, so for the rest of the argument we assume that the lifts exist globally. Let $\tilde{x}=(f, \psi)$ and $\tilde{g}$  be those lifts. Since $g_U \in \left(\coprod_{\gamma \in R_{\sigma}} \tilde{\mathfrak{S}}''_{\gamma}\right) \backslash p_{\sigma}^{-1}\left(\mathfrak{E}_{\sigma-1} \right)(U)$ and $\left(\coprod_{\gamma \in R_{\sigma}} \tilde{\mathfrak{S}}''_{\gamma}\right) \backslash p_{\sigma}^{-1}\left( \mathfrak{E}_{\sigma-1}\right) \lra \mathfrak{B}_{\sigma}$ is an open immersion, we may assume that $\tilde{g}|_U = (\mathrm{id}_{\tilde{x}_{U}}, \mathrm{id}_{N_{\sigma}})$ and similarly $\tilde{g}|_Z = (\mathrm{id}_{\tilde{x}_{Z}}, \varphi)$ for some $\varphi \in O(N_{\sigma}, \gamma)$. We now consider the associated morphisms $\tilde{x}, g(\tilde{x}) \colon S \lra \coprod_{\gamma \in R_{\sigma}} \CS''_{\gamma}(S)$. Then we have for any $\Lambda, \alpha$ that $\tilde{x}^{-1}\left(\CM_{\gamma}^{\prime \Lambda, \alpha} \right) = \tilde{x}^{\prime -1}\left(\CM_{\gamma}^{\prime \Lambda, \alpha} \right)$ because the images of a point $s \in S$ under $\tilde{x}$ and $g(\tilde{x})$ only differ by the action of $O(N_{\sigma}, \gamma)$. We find $\tilde{x}|_{\tilde{x}^{-1}\left(\CM_{\gamma}^{\prime \Lambda, \alpha} \right)} = g(\tilde{x})|_{\tilde{x}^{-1}\left(\CM_{\gamma}^{\prime \Lambda, \alpha} \right)}$ since these morphisms agree on dense open subsets, ${\tilde{x}^{-1}\left(\CM_{\gamma}^{\prime \Lambda, \alpha} \right)}$ is reduced and $\CM_{\gamma}^{\prime \Lambda, \alpha}$ is separated. It follows that $\tilde{x} = g(\tilde{x})$ and since objects in $ \CS''_{\gamma}(S)$ only have trivial automorphisms it follows that $\tilde{g}= \operatorname{id}_{\tilde{x}}$ and successively $g= \operatorname{id}_x$.

Since we have shown that the canonical morphism $\coprod_{\gamma \in R_{\sigma}} \tilde{\mathfrak{S}}''_{\gamma} \lra \tilde{\mathfrak{E}}_{\sigma}$ is finite, it follows from Remark \ref{w} that for each $\gamma \in R_{\sigma+1}$ and $j \in R_{\sigma, \sigma+1}$ the canonical morphism $W^{\gamma}_j \lra \tilde{\mathfrak{E}}_{\sigma}$ is finite. \end{proof}
Again, there exists a coarse moduli scheme $\tilde{\CE}_{\sigma}$ for $\tilde{\mathfrak{E}}_{\sigma}$. 

\begin{proposition}
There exists a $\mathbb{F}_p$-scheme $\gls{Etc''_s}$ that is locally of finite type, a surjective morphism
\begin{align*}
    \eta_{\sigma} \colon \tilde{\mathfrak{E}}_{\sigma} \lra \tilde{\CE}_{\sigma}
\end{align*}
and closed subschemes $V_0, \ldots, V_{\sigma} \subseteq \tilde{\CE}_{\sigma}$ such that $\eta_{\sigma}^{-1}\left(V_l \backslash V_{l-1}\right)$ corresponds to the substack of $\mathfrak{E}_{\sigma}$ that parametrizes families of supersingular K3 surfaces with constant Artin invariant $l$ and the restrictions
\begin{equation*}
    \eta_{\sigma}|_{\eta_{\sigma}^{-1}\left(V_l \backslash V_{l-1}\right)} \colon \eta_{\sigma}^{-1}\left(V_l \backslash V_{l-1}\right) \lra V_l \backslash V_{l-1}
\end{equation*}
are isomorphisms.
\end{proposition}

\begin{proof}
We first construct the $\tilde{\CE}_l$. We set $\tilde{\CE}_1 = \tilde{\mathfrak{E}}_1$ and consider the Ferrand pushout diagram $\tilde{\CE}_1 \leftarrow p_2^{-1}(\tilde{\CE}_1 ) \rightarrow \coprod_{\gamma \in R_2} \tilde{\CS}''_{\gamma}$ where $p_2^{-1}(\tilde{\CE}_1)$ is the closed image of $p_2^{-1}(\mathfrak{B}_{atl}^1)$ in $\coprod_{\gamma \in R_2} \tilde{\CS}''_{\gamma}$ (using the notation from the proof of Theorem \ref{theorem1}) and the morphism
\begin{equation*}
    p_2 \colon p_2^{-1}(\tilde{\CE}_1 )  \rightarrow \tilde{\CE}_1
\end{equation*}
comes from the $O(N_2, \gamma)$-equivariant morphism
\begin{equation*}
     p_2^{-1}(\mathfrak{A}_{atl}^1) \lra p_2^{-1}(\tilde{\mathfrak{E}}_1) \lra \tilde{\mathfrak{E}}_1 \lra \tilde{E}_1
\end{equation*}
and the fact that $p_2^{-1}(\mathfrak{A}_{atl}^1)  \lra  p_2^{-1}(\tilde{\CE}_1 )$ is a categorical quotient. Further, the morphism $p_2$ is finite by Lemma \ref{sweet}. Therefore, the pushout $\tilde{\CE}_2$ of $\tilde{\CE}_1 \leftarrow p_2^{-1}(\tilde{\CE}_1 ) \rightarrow \coprod_{\gamma \in R_2} \tilde{\CS}''_{\gamma}$ exists as an algebraic space locally of finite type over $\mathbb{F}_p$ \cite[Theorem A.4]{ryd2011} and by an argument analogous to the proof of the claim that $\mathfrak{B}_l$ was a scheme in the proof of Theorem \ref{theorem1} it follows that $\tilde{\CE}_2$ is a scheme. The surjective morphism of pushout diagrams
\begin{equation*}
    \left(\tilde{\mathfrak{E}}_1 \leftarrow p_2^{-1}(\tilde{\mathfrak{E}}_1 ) \rightarrow \coprod_{\gamma \in R_2} \tilde{\mathfrak{S}}''_{\gamma}\right) \lra   \left(\tilde{\CE}_1 \leftarrow p_2^{-1}(\tilde{\CE}_1 ) \rightarrow \coprod_{\gamma \in R_2} \tilde{\CS}''_{\gamma}\right)
\end{equation*}
induces the surjective morphism $\eta_2 \colon \tilde{\mathfrak{E}}_2 \lra \tilde{\CE}_2$.

Inductively, the construction of $\tilde{\CE}_l$ and $\eta_l$ goes by replacing any index $\iota$ in the previous paragraph by $\iota-2+l$. We set $V_0 = \emptyset$ and $V_l = \tilde{\CE}_l \hookrightarrow \tilde{\CE}_{\sigma}$. The remaining bit of the proposition then follows directly from Proposition \ref{coar}.
\end{proof}

\begin{proposition}
The DM-stack $\tilde{\mathfrak{E}}_{\sigma}$ over $\mathbb{F}_p$ is an algebraic space.
\end{proposition}
\begin{proof}
By the previous proposition the geometric points of $\tilde{\mathfrak{E}}_{\sigma}$ have trivial automorphism groups. The proposition then follows from \cite[Theorem 2.2.5]{conrad_2007}.
\end{proof}

We now consider the functor of sets
\begin{align*} \underline{\tilde{\CE}}_{\sigma}\colon \CA_{\mathbb{F}_p}^{\mathrm{op}} &\lra \left(\text{Sets} \right) \\
S &\longmapsto \left\{\begin{array}{l}\text{Isomorphism classes of tuples $(f, \gamma)$, where} \\ \text{$f \colon \CX \rightarrow S$ is a family of supersingular K3 surfaces} \\
\text{that fppf locally admit an $N_{\sigma}$-marking} \\
\text{and $\gamma$ is an embedding $\gamma \colon \underline{\Gamma}(2) \hookrightarrow \Pic_{\CX/S}$} \\
\text{such that for each geometric point $s \in S$} \\
\text{the sublattice $\gamma_s(\Gamma(2)) \hookrightarrow \mathrm{NS}(\CX_s)$} \\
\text{contains an ample line bundle} \\
\text{and $\gamma_s(\Gamma(2))^{\perp} \hookrightarrow \mathrm{NS}(\CX_s)$ contains no $-2$-vector}\end{array} \right\}.  \end{align*}

\begin{theorem}\label{theorem2}
The scheme $\tilde{\CE}_{\sigma}$ represents the functor $\underline{\tilde{\CE}}_{\sigma}$.
\end{theorem}
\begin{proof}
It suffices to show that the morphism $\eta_{\sigma} \colon \tilde{\mathfrak{E}}_{\sigma} \lra \tilde{\CE}_{\sigma}$ is an isomorphism of algebraic spaces. 

\claim $\eta_{\sigma}$ is finite.
\prfclaim Since the finite morphism $\coprod_{\gamma \in R_1} \CS''_{\gamma} \lra \tilde{\CE}_1$ factorizes over the finite surjection $\coprod_{\gamma \in R_1} \CS''_{\gamma} \lra \tilde{\mathfrak{E}}_1$ it follows from Lemma \ref{sweet} that $\eta_1$ is finite. Inductively, we may assume $\eta_{\sigma-1}$ is finite. Thus we have a finite morphism $\tilde{\mathfrak{E}}_{\sigma-1} \amalg \coprod_{\gamma \in R_{\sigma}} \CS_{\gamma}'' \lra \tilde{\CE}_{\sigma}$ that factorizes over a finite surjection $\tilde{\mathfrak{E}}_{\sigma-1} \amalg \coprod_{\gamma \in R_{\sigma}} \CS_{\gamma}'' \lra \tilde{\mathfrak{E}}_{\sigma}$ and by Lemma \ref{sweet} it follows that $\eta_{\sigma}$ is finite.

Since $\eta_{\sigma}$ is surjective and an isomorphism on fibers it follows that $\eta_{\sigma}$ is a bijective closed immersion. Since the restriction of $\eta_{\sigma}$ to the preimage of the open dense subset $V_{\sigma} \subseteq \tilde{\CE}_{\sigma}$ is an isomorphism it follows that $\eta_{\sigma}$ is an isomorphism.
\end{proof}

Étale under $\tilde{\CE}_{\sigma}$ there will again lie a nice scheme. However, this scheme may not be quasi-projective anymore and we introduce the following slightly weaker finiteness property. 
\begin{definition}\cite[Definition B.1]{MR3084720} A scheme $X$ is called an \emph{AF scheme} if for every finite subset $\{x_i\}$ of $X$ there exists an affine open subscheme $U$ in $X$ such that $\{x_i\}$ is contained in $U$.
\end{definition}
\begin{remark} Any quasi-projective scheme over a field $k$ is AF. Further, if $X$ is an AF scheme and $G$ is a finite group acting on $X$, then the quotient $X/G$ always exists as a scheme.
\end{remark}
\begin{remark} To our knowledge, the term \emph{AF scheme} was first used in \cite{MR3084720}. However, schemes with this property have been studied before \cite[Exp.\ V]{SGA3}, \cite[§4]{MR289501},\cite{MR2044495}. For more facts on AF schemes see \cite[Appendix B]{MR3084720}. \end{remark}
We introduce the following functors
    \begin{align*} {\underline{\CQ}_{\sigma}}^{str} \colon \CA_{\mathbb{F}_p}^{\mathrm{op}} &\lra \left(\text{Sets} \right) \\
S &\longmapsto \left\{\begin{array}{l} \text{Isomorphism classes of tuples $([G], \gamma)$, where} \\
\text{$\gamma \colon \Gamma(2) \hookrightarrow N_{\sigma}$ without $(-2)$-vector in the complement} \\
\text{and $[G]$ is an equivalence class of} \\
\text{strictly characteristic generatrices $G \subseteq pN_{\sigma}^{\vee}/pN_{\sigma} \otimes \CO_S$}\\
\text{under the $O(N_{\sigma}, \gamma)$-action} \\
\text{such that $\gamma(\Gamma(2))^{\perp} \cap \Delta_{N_{\sigma}(s)} = \emptyset$} \end{array} \right\}
\end{align*}
\begin{proposition}\label{per} There exists a separated $\mathbb{F}_p$-scheme $\gls{Q_s}$ which is of finite type and AF, and an étale surjective morphism $\pi^{\tilde{E}}_{\sigma} \colon \tilde{\CE}_{\sigma} \rightarrow \CQ_{\sigma}$ such that for $l \leq \sigma$ the image of $V_l \subseteq \tilde{\CE}_{\sigma}$ in $\CQ_{\sigma}$ is locally closed and represents the functor  ${\underline{\CQ}_{l}}^{str}$.\end{proposition}
\begin{proof} For $\sigma=1$ we can take the quasi-projective scheme $\CQ_{\sigma}= \coprod_{\gamma \in R_1} \widetilde{\CM}''_{\gamma}$. This proves the assertion in this case.

We now do induction over $\sigma$ and assume that the assertion is true for $\sigma-1$. The pushout diagram of $\mathbb{F}_p$-schemes spaces 
\begin{align*} \xymatrix{p^{-1}_{\sigma}(\tilde{\CE}_{\sigma-1}) \ar[r]_{\iota} \ar[d]_{p_{\sigma}}\po & \coprod_{\gamma \in R_{\sigma}} \tilde{\CS}''_{\gamma}\ar[d]\\
    \tilde{\CE}_{\sigma-1} \ar[r]& \tilde{\CE}_{\sigma}}\end{align*}
 induces a pushout diagram of separated $\mathbb{F}_p$-schemes of finite type and AF
 \begin{align*} \xymatrix{p^{-1}_{\sigma}(\CQ_{\sigma-1}) \ar[r]_{\iota} \ar[d]_{p_{\sigma}}\po & \coprod_{\gamma \in R_{\sigma}} \widetilde{\CM}''_{\gamma}\ar[d]\\
    \CQ_{\sigma-1} \ar[r]& \CQ_{\sigma}}\end{align*}
together with an étale and surjective morphism of pushout data 
\begin{align*} \left(\tilde{\CE}_{\sigma-1} \leftarrow p^{-1}_{\sigma}(\tilde{\CE}_{\sigma-1}) \rightarrow \coprod_{\gamma \in R_{\sigma}} \tilde{\CS}''_{\gamma}\right) \lra \left(\CQ_{\sigma-1} \leftarrow p^{-1}_{\sigma}(\CQ_{\sigma-1}) \rightarrow \coprod_{\gamma \in R_{\sigma}} \widetilde{\CM}''_{\gamma}\right). \end{align*}
By \cite[Théorème 5.4.]{MR2044495}, the pushout $\CQ_{\sigma}$ exists as an AF scheme and the induced morphism $\CQ_{\sigma-1} \amalg \coprod_{\gamma \in R_{\sigma}} \widetilde{\CM}''_{\gamma} \rightarrow \CQ_{\sigma}$ is finite surjective. Since $\coprod_{\gamma \in R_{\sigma}} \widetilde{\CM}''_{\gamma}$ is of finite type over $\mathbb{F}_p$ and we may inductively assume that $\CQ_{\sigma-1}$ is of finite type over $\mathbb{F}_p$ as well, it follows that $\CQ_{\sigma}$ is of finite type over $\mathbb{F}_p$. That $\CQ_{\sigma}$ is separated follows from \cite[Theorem 6.8.]{MR3572553} and by \cite[Theorem 6.4.]{MR3572553} the induced morphism $\tilde{\CE}_{\sigma} \rightarrow \CQ_{\sigma}$ is étale and surjective.  The assertion on the functor that is represented by the locally closed subsets $\pi^{\tilde{E}}_{\sigma}(V_l)$ can be proved in the same way as Proposition \ref{coar}. \end{proof}
\begin{remark} We will prove in Section \ref{tor} that the scheme $\CQ_{\sigma}$ constructed in the proof of Proposition \ref{per} is a coarse moduli scheme for supersingular Enriques surfaces. \end{remark}
\section{From \texorpdfstring{$\Gamma(2)$-marked}{marked} K3 surfaces to \texorpdfstring{$\Gamma'$-marked}{marked} Enriques surfaces}\label{enr}
Although we want to construct a moduli space for Enriques surfaces, we have only discussed K3 surfaces so far. In this section we establish the connection between $\Gamma(2)$-marked supersingular K3 surfaces and $\Gamma'$-marked Enriques surfaces that are quotients of supersingular K3 surfaces. 

\begin{definition} If $X$ is a supersingular K3 surface and $\iota \colon X \rightarrow X$ is a fixed point free involution, we write $G=\langle \iota \rangle$ for the cyclic group of order $2$ which is generated by $\iota$. A quotient of surfaces $X \rightarrow X/G=Y$ defined by such a pair $(X, \iota)$ is called a \textit{supersingular Enriques surface} $Y$. The \emph{Artin invariant} of a supersingular Enriques surface $Y$ is the Artin invariant of the supersingular K3 surface $X$ that universally covers $Y$. A \textit{family of supersingular Enriques surfaces} is a smooth and proper morphism of algebraic spaces $f \colon \CY \rightarrow S$ over $\mathbb{F}_p$ such that for each field $k$ and each $s \colon \Spec k \rightarrow S$ the fiber $f_s \colon \CY_s \rightarrow \Spec k$ is a supersingular Enriques surface. \end{definition}

Recall from Section \ref{aux} that we defined $\Gamma$ to be the lattice $\Gamma= U_2 \oplus E_8(-1)$. If $Y$ is a supersingular Enriques surface, then there exists an isomorphism of lattices \mbox{$\mathrm{Pic}(Y) \cong \Gamma \oplus \mathbb{Z}/2\mathbb{Z}$} and we denote the latter lattice by $\gls{Gamma'}$. In arbitrary characteristic, by \cite[Proposition 4.4]{MR3393362}, if $\CY \rightarrow S$ is a family of supersingular Enriques surfaces, then the torsion part $\Pic_{\CY/S}^{\tau}$ of the Picard scheme is a finite flat group scheme of length $2$ over $S$. In particular, when $p \geq 3$ we have an equality of sheaves of groups $\Pic_{\CY/S}^{\tau} = \underline{\mathbb{Z}/2\mathbb{Z}}$ with generator $\omega_{\CY/S}$. Further, in arbitrary characteristic, the quotient $\Pic_{\CY/S}/ \Pic_{\CY/S}^{\tau}$ is a locally constant sheaf of torsion-free finitely generated abelian groups. In characteristic $p \geq 3$ this implies that there exists an étale covering $\{U_i \rightarrow S\}_{i \in I}$ such that we have an isomorphism $\Pic_{\CY_{U_i}/U_i} \cong \underline{\Gamma \oplus \mathbb{Z}/2\mathbb{Z}}$ for each $i \in I$. 

\begin{definition} A \emph{$\Gamma$-marking} of a family $f \colon \CY \rightarrow S$ of supersingular Enriques surfaces is the choice of a morphism $\tilde{\gamma} \colon \underline{\Gamma} \rightarrow \Pic_{\CY/S}$ of group objects in the category of algebraic spaces compatible with the intersection forms. Analogously we define the notion of a \emph{$\Gamma'$-marking}. There are obvious notions of morphisms of families of marked supersingular Enriques surfaces.  \end{definition}

As before, we will in the following always assume that $p \neq 2$. In general, it is not true that a family of $\Gamma'$-marked supersingular Enriques surfaces $\mathcal{Y} \rightarrow S$ has a canonical family of supersingular K3 surfaces that covers it. Depending on $S$ such a family might not exist and even if it exists it might not be unique. For details we refer to \cite[Section 5]{schrer2020enriques}. We therefore enrich our families of supersingular Enriques surfaces with a datum that gives a canonical K3 cover.
\begin{definition} Let $\mathcal{Y} \rightarrow S$ be a family of supersingular Enriques surfaces. A \emph{canonizing datum} on $\mathcal{Y}$ is a pair $(\mathcal{L}, \mu)$, where $\mathcal{L} \in \Pic(\mathcal{Y})$ is a line bundle such that for each field $k$ and each $s \colon \Spec k \rightarrow S$ there is an isomorphism $\mathcal{L}|_{\mathcal{Y}_s} \cong \omega_{\mathcal{Y}_s}$ and $\mu$ is an isomorphism $\mu \colon \mathcal{L} \otimes \mathcal{L} \rightarrow \CO_{\CY}$. A \emph{family of $\Gamma'$ marked supersingular Enriques surfaces with canonizing datum} is a quadruple $(\tilde{f} \colon \CY \rightarrow S, \tilde{\gamma} \colon \Gamma' \rightarrow \Pic_{\CY/S}, \CL, \mu)$ such that $(\CL, \mu)$ is a canonizing datum. There is an obvious notion of isomorphisms of such families.
\end{definition} 
\begin{remark} Let $Y \rightarrow \Spec k$ be a supersingular Enriques surface. Then the canonical bundle $\omega_{Y}$ together with the choice of an isomorphism $\mu \colon \omega_{Y} \otimes \omega_{Y} \rightarrow \CO_Y$ yields a canonizing datum $(\omega_{Y/k}, \mu)$ on $Y$ and this datum up to isomorphism the only canonizing datum on $Y$.
\end{remark}
\begin{proposition}[and Definition] Given a family of $\Gamma'$-marked supersingular Enriques surfaces with canonizing datum $(\tilde{f} \colon \CY \rightarrow S, \tilde{\gamma} \colon \underline{\Gamma}' \rightarrow \Pic_{\CY/S}, \CL, \mu)$ we let $\CX = \underline{\Spec}_{\CY}(\CO_Y \oplus \CL)$. Then $\CX$ is a family of supersingular K3 surfaces $f \colon \CX \rightarrow S$ together with a morphism $\CX \rightarrow \CY$ that is a finite étale cover of $\CY$ of degree $2$. Further, this family carries a canonical $\Gamma(2)$-marking $\gamma \colon \underline{\Gamma}(2) \hookrightarrow \Pic_{\CX/S}$ induced from the $\Gamma'$-marking on $\CY$. We call $\CX \rightarrow \CY$ the \emph{canonical K3 cover of }$(\CY, \CL, \mu)$. \end{proposition} 
\begin{proof} The canonical morphism $\CX \rightarrow \CY$ is finite of degree $2$ and since $p \neq 2$ it is also étale. It follows that $\CX \rightarrow S$ is proper and smooth. Further, every fiber $\CX_s \rightarrow \CY_s$ is just the universal K3 cover of  the Enriques surface $\CY_s$ and it follows that $\CX \rightarrow S$ is a family of supersingular K3 surfaces. 

Pullback of line bundles induces a morphism $\Pic_{\CY/S} \rightarrow \Pic_{\CX/S}$ of group objects in the category of algebraic spaces over $S$, and because the morphism $\CX \rightarrow \CY$ is unramified and $2$-to-$1$, the intersection form under this morphism gets multiplied by $2$. In other words, after twisting the intersection form of $\Pic_{\CY/S}$ by the factor $2$, we obtain a morphism $\Pic_{\CY/S}(2) \rightarrow \Pic_{\CX/S}$ of group objects in the category of algebraic spaces over $S$ compatible with intersection forms. Now precomposing with the marking $\tilde{\psi}_{|_{\Gamma}}(2) \colon \underline{\Gamma}(2) \rightarrow \Pic_{\CY/S}(2)$ yields an embedding $\gamma \colon \underline{\Gamma}(2) \hookrightarrow \Pic_{\CX/S}$. \end{proof}
Next, we show that any $\Gamma$-marking on a family of supersingular Enriques surfaces extends in a unique way to a $\Gamma'$-marking.
\begin{lemma}\label{equivgamma} Let $S$ be an algebraic space over $\mathbb{F}_p$. The forgetful functor
\begin{align*} \left\{\begin{array}{l}\text{Families of $\Gamma'$-marked} \\ \text{supersingular Enriques surfaces} \\ \text{$(\tilde{f} \colon \CY \rightarrow S, \tilde{\gamma} \colon \Gamma' \rightarrow \Pic_{\CY/S})$} \end{array} \right\} \lra  \left\{\begin{array}{l}\text{Families of $\Gamma$-marked} \\ \text{supersingular Enriques surfaces} \\ \text{$(\tilde{f} \colon \CY \rightarrow S, \tilde{\gamma}_{|_{\Gamma}} \colon \Gamma \rightarrow \Pic_{\CY/S})$} \end{array} \right\}\end{align*}
is an equivalence of categories. \end{lemma}
\begin{proof} The automorphism group of the constant group scheme $\underline{\mathbb{Z}/2\mathbb{Z}}$ is trivial. Thus, every $\Gamma$-marking extends étale locally in a unique way to a $\Gamma'$-marking and by uniqueness to a global $\Gamma'$-marking. \end{proof}
\begin{remark} The result from the lemma obviously extends to families of marked supersingular Enriques surfaces with canonizing datum.
\end{remark}
We now consider the stack
\begin{align*} \mathfrak{E}_{\sigma} \colon \CA_{\mathbb{F}_p}^{\mathrm{op}} &\lra \left(\text{Sets} \right) \\
S &\longmapsto \left\{\begin{array}{l}\text{Tuples $(\tilde{f}, \tilde{\gamma}, \CL, \mu)$, where} \\ \text{$\tilde{f} \colon \CY \rightarrow S$ is a family of supersingular Enriques surfaces} \\  \text{with $\Gamma'$-marking $\tilde{\gamma} \colon \Gamma' \rightarrow \Pic_{\CY/S}$ and canonizing datum $(\CL, \mu)$} \\
\text{such that the canonical K3 cover $\CX \rightarrow \CY$} \\
\text{fppf locally admits an $N_{\sigma}$-marking} \end{array} \right\}.\end{align*}
We are interested in the representability of the moduli stack $\mathfrak{E}_{\sigma}$. In the following proposition we show that the stack $\mathfrak{E}_{\sigma}$ is isomorphic to the functor $\tilde{\mathfrak{E}}_{\sigma}$ from Section \ref{secmark}. 
\begin{proposition}\label{EandK} There exists an isomorphism of stacks $\mathrm{cov} \colon \mathfrak{E}_{\sigma} \rightarrow \tilde{\mathfrak{E}}_{\sigma}$. \end{proposition}
\begin{proof} We first define the morphism $\mathrm{cov} \colon \mathfrak{E}_{\sigma} \rightarrow \tilde{\mathfrak{E}}_{\sigma}$. To this end, we consider a family of $\Gamma'$-marked supersingular Enriques surfaces with canonizing datum $y=(\tilde{f} \colon \CY \rightarrow S, \tilde{\gamma} \colon \Gamma' \rightarrow \Pic_{\CY/S}, \CL, \mu) \in \mathfrak{E}_{\sigma}(S)$ that has the canonical K3 cover $(f \colon \CX \rightarrow S, \gamma \colon \Gamma(2) \hookrightarrow \Pic_{\CX/S})$. If $s \colon \Spec \overline{k} \rightarrow S$ is a geometric point, then the orthogonal complement of $\gamma_s(\Gamma(2))$ in $\mathrm{NS}(\CX_s)$ contains no $(-2)$-vector. Since the fiber $\CY_s$ is projective, it has an ample divisor. Pullback along finite morphisms preserves ampleness of divisors, so the sublattice $\gamma_s(\Gamma(2)) \hookrightarrow \mathrm{NS}(\CX_s)$ also contains an ample divisor. We can thus define $\mathrm{cov}(S)(y)=(f \colon \CX \rightarrow S, \gamma \colon \underline{\Gamma}(2) \hookrightarrow \Pic_{\CX/S})$. The definition of $\mathrm{cov}(S)$ on morphisms is clear and we obtain a morphism of stacks.

We will now define another morphism of stacks $\mathrm{quot} \colon \tilde{\mathfrak{E}}_{\sigma} \rightarrow \mathfrak{E}_{\sigma}$ such that the morphisms $\mathrm{quot}$ and $\mathrm{cov}$ are mutually quasi-inverse to each other. To this end, we let $S$ be a scheme and let $x=(f \colon \CX \rightarrow S, \gamma \colon \Gamma(2) \hookrightarrow \Pic_{\CX/S}) \in \tilde{\mathfrak{E}}_{\sigma}(S)$. We consider the involution $\iota_{\gamma} \colon \CX \rightarrow \CX$ from the proof of Proposition \ref{amp}. Then $\iota_{\gamma}$ induces a free $\langle \iota_{\gamma} \rangle$-action on $\CX$ and thus the quotient $\CY= \CX/ \langle \iota_{\gamma} \rangle$ exists as an algebraic space over $S$ and the morphism $c \colon \CX \rightarrow \CY$ makes $\CX$ into a $\mathbb{Z}/2\mathbb{Z}$-torsor over $\CY$. Thus, for every $s \in S$, $\CX_s$ is a $\mathbb{Z}/2\mathbb{Z}$-torsor over $\CY_s$ and it follows that $\CY_s$ is a supersingular Enriques surface for each $s \in S$. Further, the canonical morphism $\Pic_{\CY/S} \rightarrow \Pic_{\CX/S}$ induces an isomorphism $\psi \colon \Pic_{\CY/S}(2) \lra \gamma(\underline{\Gamma}(2))$. We define $\tilde{\gamma} \colon \Gamma' \rightarrow \Pic_{\CY/S}$ to be the unique $\Gamma'$-marking of $\Pic_{\CY/S}$ which is induced from $\psi^{-1}$ using Lemma \ref{equivgamma}. 

Since $c$ is finite flat of degree $2$, the $\CO_{\CY}$-algebra $c_{\ast}\CO_{\CX}$ is a locally free $\CO_{\CY}$-module of rank $2$. We set $\CL = c_{\ast}\CO_{\CX}/\CO_{Y}$. Then $\CL$ is a line bundle on $\CY$ with $\CL|_{\CY_s} \cong \omega_{\CY_s}$ for all $s \in S$. Because $c$ is étale, the multiplication map $\CL \times \CL \rightarrow c_{\ast}\CO_{\CX}/ \CL \cong \CO_{\CY}$ is surjective and we can take $\mu \colon \CL \otimes \CL \rightarrow \CO_{\CY}$ to be the induced isomorphism. 

Now setting $\mathrm{quot}(S)(x)= (\tilde{f} \colon \CY \rightarrow S, \tilde{\gamma} \colon \Gamma' \rightarrow \Pic_{\CY/S}, \CL, \mu)$ yields the desired inverse. \end{proof}
The functor of sets corresponding to $\mathfrak{E}_{\sigma}$ is
\begin{align*} \underline{\CE}_{\sigma} \colon \CA_{\mathbb{F}_p}^{\mathrm{op}} &\lra \left(\text{Sets} \right) \\
S &\longmapsto \left\{\begin{array}{l}\text{Isomorphism classes of tuples $(\tilde{f}, \tilde{\gamma}, \CL, \mu)$, where} \\ \text{$\tilde{f} \colon \CY \rightarrow S$ is a family of supersingular Enriques surfaces} \\  \text{with $\Gamma'$-marking $\tilde{\gamma} \colon \Gamma' \rightarrow \Pic_{\CY/S}$ and canonizing datum $(\CL, \mu)$} \\
\text{such that the canonical K3 cover $\CX \rightarrow \CY$} \\
\text{fppf locally admits an $N_{\sigma}$-marking} \end{array} \right\}.\end{align*}
The following theorem, which is one of the main results in this work, can be seen as a supersingular version of the results on complex Enriques surfaces in \cite{MR771979} or as a version for Enriques surfaces of the results on supersingular K3 surfaces in \cite{MR717616}.
\begin{theorem}\label{main} The functor $\underline{\CE}_{\sigma}$ is represented by a scheme $\gls{E_s}$ which is locally of finite type over $\mathbb{F}_p$ and there exists an étale surjective morphism $\pi^E_{\sigma} \colon \CE_{\sigma} \rightarrow \CQ_{\sigma}$.  \end{theorem}
\begin{proof} This follows directly from Theorem \ref{theorem2}, Proposition \ref{per} and Proposition \ref{EandK}. \end{proof}
\begin{remark} It follows from \cite[Proposition 3.5]{2013arXiv1301.1118J} that for any $\sigma \geq 5$ we have a canonical isomorphism $\mathfrak{E}_{\sigma} \stackrel{\sim}\lra \mathfrak{E}_5$. \end{remark}
The previous remark motivates the following definition.
\begin{definition} We call $\CE_{5}$ the \emph{moduli space} of $\Gamma'$-marked supersingular Enriques surfaces and $\CQ_{5}$ the \emph{period space} of $\Gamma'$-marked supersingular Enriques surfaces. \end{definition}
\begin{remark} From the constructions it follows directly that, similar to the case of marked supersingular K3 surfaces, there are canonical stratifications $\CE_1 \hookrightarrow \CE_2 \hookrightarrow \CE_3 \hookrightarrow \CE_4 \hookrightarrow \CE_5$ and $\CQ_1 \hookrightarrow \CQ_2 \hookrightarrow \CQ_3 \hookrightarrow \CQ_4 \hookrightarrow \CQ_5$ via closed immersions. However, the latter are not sections to fibrations of the form $\CQ_{\sigma} \rightarrow \CQ_{\sigma-1}$. The main difference to the situation for marked supersingular K3 surfaces, and therefore the reason why such a fibration does not exist, is the following. While the embedding $\CM_{\sigma-1} \hookrightarrow \CM_{\sigma}$ depends on the choice of an embedding $j \colon N_{\sigma} \hookrightarrow N_{\sigma-1}$, the embedding $\CQ_{\sigma-1} \hookrightarrow \CQ_{\sigma}$ corresponds to the union over all images of such embeddings $\CM_{\sigma-1} \hookrightarrow \CM_{\sigma}$, but the inclusion $\bigcup_{j \in R_{\sigma-1,\sigma}} \Phi_j(\CM_{\sigma-1}) \hookrightarrow \CM_{\sigma}$ does not have an inverse.
\end{remark}
\begin{remark} The period spaces $\CQ_{\sigma}$ come with canonical compactifications which we denote $\gls{Qtd_s}$. Namely, we consider the functor 
\begin{align*} \tilde{\CE}_{\sigma}^{\dag} \colon \CA_{\mathbb{F}_p}^{\mathrm{op}} &\lra \left(\text{Sets} \right) \\
S &\longmapsto \left\{\begin{array}{l}\text{Isomorphism classes of tuples $(f, \gamma)$, where} \\ \text{$f \colon \CX \rightarrow S$ is a family of supersingular K3 surfaces} \\
\text{that fppf locally admit an $N_{\sigma}$-marking} \\
\text{and $\gamma$ is an embedding $\gamma \colon \underline{\Gamma}(2) \hookrightarrow \Pic_{\CX/S}$} \\
\text{such that for each geometric point $s \in S$} \\
\text{the sublattice $\gamma_s(\Gamma(2)) \hookrightarrow \mathrm{NS}(\CX_s)$} \\
\text{contains an ample line bundle} \end{array} \right\}.  \end{align*}
By an argument analogous to the proof of Theorem \ref{theorem1} it follows that the functor $\tilde{\CE}^{\dag}_{\sigma}$ is representable by a scheme $\gls{Etd_s}$ which is locally of finite type over $\mathbb{F}_p$. Further, there exists a proper $\mathbb{F}_p$-scheme $\CQ_{\sigma}^{\dag}$ and a canonical étale surjective morphism $\tilde{\CE}_{\sigma}^{\dag} \rightarrow \CQ_{\sigma}^{\dag}$ by an argument analogous to the one in the proof of Proposition \ref{per}. 

The scheme $\CQ_{\sigma}^{\dag}$ is indeed proper because inductively there exists a finite surjection of the proper $\mathbb{F}_p$-scheme $\CQ_{\sigma-1}^{\dag} \amalg \coprod_{\gamma \in R_{\sigma}} \widetilde{\CM}'_{\gamma}$ onto $\CQ_{\sigma}^{\dag}$. The canonical morphism of schemes $\CQ_{\sigma} \rightarrow \CQ_{\sigma}^{\dag}$ is an open immersion and a subscheme of the closed locus $\CQ_{\sigma}^{\dag} \backslash \CQ_{\sigma}$ corresponds to quotients of K3 surfaces by involutions that fix a divisor. This is an analogue to the so-called Coble locus in the characteristic zero setting, see \cite{MR3098788}. \end{remark}
\section{Some remarks about the geometry of the moduli space \texorpdfstring{$\CE_{\sigma}$}{}}\label{geom}
The geometry of $\CE_{\sigma}$ is quite complicated. We have shown that the scheme $\CE_{\sigma}$ is reduced, but in general it will not be connected, since already in the case $\sigma=1$ it has multiple connected components.

Moreover, we can not expect the connected components of $\CE_{\sigma}$ to be irreducible, since they are glued together from the schemes $\tilde{\CS}''_{\gamma}$ with $\gamma \in R_{\sigma}$ and we can not expect the irreducible components to be smooth: a priori the action of $O(N_{\sigma},\gamma)$ on $\CS'_{\gamma}$ which we took the quotient by is not free and we do not expect it to factorize over a free action. 

Further, when taking the pushout in the proof of Theorem \ref{theorem1}, we expect more singularities to show up. However, there are some simple general observations on the geometry of the scheme $\CE_{\sigma}$.

 We will first introduce a subfunctor $\tilde{\CE}'_{\sigma}$ of $\tilde{\CE}_{\sigma}$ to help us understand the geometry of the scheme $\CE_{\sigma} \cong \tilde{\CE}_{\sigma}$. We define
\begin{align*} \tilde{\CE}'_{\sigma} \colon \CA_{\mathbb{F}_p}^{\mathrm{op}} &\lra \left(\text{Sets} \right) \\
S &\longmapsto \left\{\begin{array}{l}\text{Isomorphism classes of tuples $(f, \gamma)$, where} \\ \text{$f \colon \CX \rightarrow S$ is a family of supersingular K3 surfaces} \\
\text{that fppf locally admit an $N_{\sigma}$-marking} \\
\text{and $\gamma$ is an embedding $\gamma \colon \underline{\Gamma}(2) \hookrightarrow \Pic_{\CX/S}$} \\
\text{such that $\gamma(\underline{\Gamma}(2)) \subset \psi(\underline{N}_{\sigma})$ for some} \\
\text{fppf local marking $\psi \colon \underline{N}_{\sigma} \hookrightarrow \Pic_{\CX/S}$ and} \\
\text{such that for each geometric point $s \in S$} \\
\text{the sublattice $\gamma_s(\Gamma(2)) \hookrightarrow \mathrm{NS}(\CX_s)$} \\
\text{contains an ample line bundle} \end{array} \right\}.  \end{align*}
The proof of the following proposition goes similarly to the proof of Theorem \ref{theorem1}. We therefore only highlight the main differences in the proof.
\begin{proposition} The functor $\tilde{\CE}'_{\sigma}$ is representable by a closed algebraic subspace $\gls{Et'_s}$ of $\tilde{\CE}_{\sigma}$.
\end{proposition}
\begin{proof}
We do induction over $\sigma$. The case $\sigma=1$ is clear, because in this case we have $\tilde{\CE}_{1}= \tilde{\CE}'_1$. 

We write $\tilde{\CE}'^s_{\sigma-1}$ for the subfunctor of $\tilde{\CE}'_{\sigma-1}$ which is defined to be as follows: the $S$-valued points of $\tilde{\CE}'^s_{\sigma-1}$ are the families $f \colon \CX \rightarrow S$ in $\tilde{\CE}'_{\sigma-1}(S)$ that admit markings of the form $\gamma \colon \underline{\Gamma}(2) \hookrightarrow N_{\sigma-1} \hookrightarrow \Pic_{\CX/S}$ such that there is a factorization \mbox{$\gamma \colon \underline{\Gamma}(2) \hookrightarrow N_{\sigma} \hookrightarrow N_{\sigma-1} \hookrightarrow \Pic_{\CX/S}$}.
 
Then $\tilde{\CE}'^s_{\sigma-1} \subset \tilde{\CE}'_{\sigma-1}$ is a closed subfunctor, since $\tilde{\CE}'^s_{\sigma-1}$ is representable by the image of the finite morphism $p_{\sigma} \colon p_{\sigma}^{-1}(\tilde{\CE}_{\sigma-1}) \rightarrow \tilde{\CE}_{\sigma-1}$. We consider the pushout diagram
\begin{align*} \xymatrix{p^{-1}_{\sigma}(\tilde{\CE}'^{s}_{\sigma-1}) \ar[r]_{\iota} \ar[d]_{p_{\sigma}}\po & \coprod_{\gamma \in R_{\sigma}} \tilde{\CS}''_{\gamma}\ar[d]\\
   \tilde{\CE}'^{s}_{\sigma-1} \ar[r]& \CP.}\end{align*}
We note that $p_{\sigma} \colon p^{-1}_{\sigma}(\tilde{\CE}'^{s}_{\sigma-1}) \rightarrow   \tilde{\CE}'^{s}_{\sigma-1}$ is finite surjective and therefore also $\coprod_{\gamma \in R_{\sigma}} \tilde{\CS}''_{\gamma} \rightarrow \CP$ is finite surjective. Analogously to the proofs of Theorem \ref{theorem1} and \ref{theorem2}  we can show that $\CP$ exists as a scheme and represents the functor $\tilde{\CE}'_{\sigma}$. Thus, we set $\tilde{\CE}'_{\sigma}= \CP$. Since $\tilde{\CE}'^{s}_{\sigma-1}$ is closed in $\tilde{\CE}_{\sigma-1}$ it follows from the construction of the scheme $\tilde{\CE}_{\sigma}$ that $\tilde{\CE}'_{\sigma}$ is a closed subscheme of $\tilde{\CE}_{\sigma}$. 
\end{proof}
Again, the functor $\tilde{\CE}'_{\sigma}$ has a description in terms of Enriques surfaces. Namely, we define
\begin{align*} \underline{\CE}'_{\sigma} \colon \CA_{\mathbb{F}_p}^{\mathrm{op}} &\lra \left(\text{Sets} \right) \\
S &\longmapsto \left\{\begin{array}{l}\text{Isomorphism classes of tuples $(\tilde{f}, \tilde{\gamma}, \CL, \mu)$, where} \\ \text{$\tilde{f} \colon \CY \rightarrow S$ is a family of supersingular Enriques surfaces} \\  \text{with $\Gamma'$-marking $\tilde{\gamma} \colon \Gamma' \rightarrow \Pic_{\CY/S}$ and canonizing datum $(\CL, \mu)$} \\
\text{such that the canonical K3 cover $\CX \rightarrow \CY$} \\
\text{fppf locally admits an $N_{\sigma}$-marking such that the induced map}  \\
\text{$\Gamma(2) \rightarrow \Pic_{\CX/S}$ factorizes through $\underline{N}_{\sigma}$} \end{array} \right\}.\end{align*}
The proof of the following proposition goes completely analogously to the proof of Proposition \ref{EandK} and we therefore leave it to the reader.
\begin{proposition} There exists an isomorphism of functors $\mathrm{cov} \colon \underline{\CE}'_{\sigma} \rightarrow \tilde{\CE}'_{\sigma}$.
\end{proposition}
We will write $\gls{E'_s}$ for the scheme representing the functor $\underline{\CE}'_{\sigma}$. Coming back to the discussion of the geometry of the space $\CE_{\sigma}$, we note that the space $\CE_{\sigma}$ is of dimension $\sigma-1$, but its irreducible components might in general not be equidimensional. The upshot of constructing the functor $\tilde{\CE}'_{\sigma}$ lies in the following result.
\begin{proposition}\label{upshot} For any $\sigma' \leq \sigma$, the scheme $\CE'_{\sigma'}$ is a closed subscheme of $\CE_{\sigma}$ and we have the equality
\begin{align*}
\bigcup_{\sigma' \leq \sigma} \CE'_{\sigma'} = \CE_{\sigma}.
\end{align*}
Further, $\CE'_{\sigma}$ is the maximal closed subspace in $\CE_{\sigma}$ with the property that all of its irreducible components are of dimension $\sigma-1$.
\end{proposition}
\begin{proof}
The first statement follows from the construction of the space $\CE_{\sigma}$ via induction over $\sigma$ and the second statement follows directly from the construction of $\CE_{\sigma}$ and $\CE'_{\sigma}$ and the fact that the morphism $\coprod_{\gamma \in R_{\sigma}} \CS''_{\gamma} \rightarrow \CE'_{\sigma}$ is a finite surjection.
\end{proof}
\begin{remark} We do not know if the functors $\underline{\CE}_{\sigma}$ and $\underline{\CE}'_{\sigma}$ are unequal in general. This boils down to asking whether there exist embeddings $\Gamma(2) \hookrightarrow N_{\sigma-1}$ that do not factorize over an embedding $j \colon N_{\sigma} \hookrightarrow N_{\sigma-1}$. However, we suspect that such embeddings may exist and that for $\sigma > 1$ we should have $\underline{\CE}_{\sigma} \neq \underline{\CE}'_{\sigma}$. \end{remark}
There exists a scheme lying under $\CE'_{\sigma}$ in analogy to Proposition \ref{per}.
\begin{proposition}
There exists a separated $\mathbb{F}_p$-scheme $\gls{Q'_s}$, which is a closed subscheme of $\CQ_{\sigma}$, and a canonical étale surjective morphism $\tilde{\CE}'_{\sigma} \rightarrow \CQ'_{\sigma}$.
\end{proposition}
\begin{proof}
The proof goes analogously to the proof of Proposition \ref{per} by replacing $\CQ_{\sigma-1}$ with the image of $p_{\sigma}^{-1}(\CQ_{\sigma-1})$ in $\CQ_{\sigma-1}$ in the pushout construction.
\end{proof}
The following proposition is an analogue to Proposition \ref{upshot}.
\begin{proposition}\label{upshotper} For any $\sigma' \leq \sigma$, the scheme $\CQ'_{\sigma}$ is a closed subscheme of $\CQ_{\sigma}$ and we have an equality
\begin{align*}
\bigcup_{\sigma' \leq \sigma} \CQ'_{\sigma'} = \CQ_{\sigma}.
\end{align*}
Further, $\CQ'_{\sigma}$ is the maximal closed subscheme in $\CQ_{\sigma}$ whose irreducible components are all of dimension $\sigma-1$.
\end{proposition}
In the following, we give some results on the geometry of the spaces $\CE'_{\sigma}$ and $\CQ'_{\sigma}$. It follows from Proposition \ref{upshot} and Proposition \ref{upshotper} that the geometry of these spaces is intimately related to the geometry of the spaces $\CE_{\sigma}$ and $\CQ_{\sigma}$. 
\begin{definition} We write $\varepsilon_{\sigma}$ for the number of irreducible components of $\CE'_{\sigma}$. \end{definition}
\begin{remark} We recall from Section \ref{secclasmark} that the $\mathbb{F}_p$-scheme $\CS_{\sigma}$ is smooth. In particular each of its connected components is irreducible. From its description as the moduli space of characteristic subspaces together with ample cones it is clear that $\CS_{\sigma}$ only has finitely many connected components. \end{remark}
\begin{proposition}\label{irredA} The morphism $p_{\sigma} \colon \coprod_{\gamma \in R_{\sigma}} \tilde{\CS}''_{\gamma} \rightarrow \CE'_{\sigma}$ induces a bijection between the sets of irreducible components of $\coprod_{\gamma \in R_{\sigma}} \tilde{\CS}''_{\gamma}$ and $\CE'_{\sigma}$. If we write $\tau_{\sigma}$ for the number of connected components of $\CS_{\sigma}$, we obtain the inequality 
$$\varepsilon_{\sigma} \leq \tau_{\sigma} \cdot |R_{\sigma}|.$$   \end{proposition}
\begin{proof} For $\gamma \in R_{\sigma}$, each irreducible component of the scheme $\tilde{\CS}''_{\gamma}$ over $\mathbb{F}_p$ is of dimension $\sigma - 1$. Since there exists a dense open subspace $U \subset \coprod_{\gamma \in R_{\sigma}} \tilde{\CS}''_{\gamma}$ such that the restriction $p_{\sigma}|_{U} \colon U \rightarrow \CE'_{\sigma}$ is an open immersion, it follows that if $E_1, E_2 \subset \coprod_{\gamma \in R_{\sigma}} \tilde{\CS}''_{\gamma}$ are two different irreducible components, then the intersection $p_{\sigma}(E_1) \cap p_{\sigma}(E_2)$ is at least of codimension $1$. Thus, the morphism $p_{\sigma}$ induces a bijection between the sets of irreducible components of $\coprod_{\gamma \in R_{\sigma}} \tilde{\CS}''_{\gamma}$ and $\CE'_{\sigma}$. The inequality follows from the fact that the open subscheme $\CS''_{\gamma} \subset \CS_{\sigma}$ surjects onto $\tilde{\CS}''_{\gamma}$ and each connected component of $\CS_{\sigma}$ is irreducible.   \end{proof}
\begin{proposition}\label{irredS} There is an equality 
$$\# \{\text{irreducible components of $\CQ'_{\sigma}$}\} = |R_{\sigma}|.$$ \end{proposition}
\begin{proof} This follows since the schemes $\widetilde{\CM}''_{\gamma}$ are irreducible and there is a dense open subscheme of $\coprod_{\gamma \in R_{\sigma}} \widetilde{\CM}''_{\gamma}$ which is isomorphic to a dense open subscheme of $\CQ'_{\sigma}$. \end{proof}
\begin{definition} On the set $R_{\sigma}$ of isomorphism classes $[\gamma \colon \Gamma(2) \hookrightarrow N_{\sigma}]$ of embeddings of lattices we define an equivalence relation via 
\begin{align*} [\gamma \colon \Gamma(2) \hookrightarrow N_{\sigma}] &\sim [\gamma' \colon \Gamma(2) \hookrightarrow N_{\sigma}] 
\intertext{if and only if there exists a positive integer $\sigma' \leq \sigma$ and embeddings $j \colon N_{\sigma} \hookrightarrow N_{\sigma'}$ and $j' \colon N_{\sigma} \hookrightarrow N_{\sigma'}$, such that the sublattice $j(\gamma(\Gamma(2)))^{\perp} \subset N_{\sigma'}$ contains no $(-2)$-vectors and such that there is an equality}
[j \circ \gamma \colon \Gamma(2) \hookrightarrow N_{\sigma'}] &= [j' \circ \gamma' \colon \Gamma(2) \hookrightarrow N_{\sigma'}]\end{align*}
of elements in $R_{\sigma'}$. \end{definition}
 Using this equivalence relation we obtain the following results.
\begin{proposition} There is an equality 
$$\# \{\text{connected components of $\CQ'_{\sigma}$}\} =  \left| R_{\sigma} / \sim \right|.$$ \end{proposition}
\begin{proof} It follows from the construction in the proof of Proposition \ref{upshotper} that under the surjection of schemes $\coprod_{\gamma \in R_{\sigma}} \widetilde{\CM}''_{\gamma} \rightarrow \CQ'_{\sigma}$ two connected components $\widetilde{\CM}''_{\gamma_1}$ and $\widetilde{\CM}''_{\gamma_2}$ map to the same connected component of $\CQ'_{\sigma}$ if and only if $\gamma_1 \sim \gamma_2$.  \end{proof}
\begin{proposition} We write $\tau_{\sigma}$ for the number of connected components of $\CS_{\sigma}$ and $\varepsilon_{\sigma}^c$ for the number of connected components of $\CE'_{\sigma}$. There is an inequality 
$$\varepsilon_{\sigma}^c \leq \tau_{\sigma} \cdot \left| R_{\sigma} / \sim \right|.$$ \end{proposition}
\begin{proof} We consider the surjection $\coprod_{\gamma \in R_{\sigma}} \tilde{\CS}''_{\gamma} \rightarrow \CE'_{\sigma}$. For each $\gamma \in R_{\sigma}$ the scheme $\tilde{\CS}''_{\gamma}$ has at most $\tau_{\sigma}$ many connected components. If $\gamma_1 \sim \gamma_2$, say with $[j_1 \circ \gamma_1]= [j_2 \circ \gamma_2]$, then $\tilde{\CS}''_{j_1 \circ \gamma_1} \cong \tilde{\CS}''_{j_2 \circ \gamma_2}$ is a subspace of both $\tilde{\CS}''_{\gamma_1}$ and $\tilde{\CS}''_{\gamma_2}$ which touches each of the connected components of the $\tilde{\CS}''_{\gamma_i}$. Thus, the image of $\tilde{\CS}''_{\gamma_1} \amalg \tilde{\CS}''_{\gamma_2}$ in $\CE'_{\sigma}$ has at most $\tau_{\sigma}$ many connected components and this implies the statement of the proposition.\end{proof}
\begin{proposition} We denote by $\alpha_{\sigma}$ the number of isomorphism classes  $[\gamma \colon \Gamma(2) \hookrightarrow N_{\sigma}]$ in $R_{\sigma}$ such that that for each positive integer $\sigma' < \sigma$ and each embedding of lattices $j \colon N_{\sigma} \hookrightarrow N_{\sigma'}$ there is a $-2$-vector in the sublattice $j(\gamma(\Gamma(2)))^{\perp} \subset N_{\sigma'}$. Then we have an inequality 
$$ \alpha_{\sigma} \leq \varepsilon_{\sigma}^c \leq \tau_{\sigma} \cdot (\alpha_{\sigma} + \varepsilon_{\sigma - 1}^c). $$\end{proposition}
\begin{proof} The lower bound is a very weak estimate: if $\gamma$ is such that for each positive integer $\sigma' < \sigma$ and each $j \colon N_{\sigma} \hookrightarrow N_{\sigma'}$ there is a $(-2)$-vector in the sublattice $j(\gamma(\Gamma(2)))^{\perp} \subset N_{\sigma'}$, then $[\gamma]$ is the only element in its equivalence class of $\sim$. Hence, the image of $\tilde{\CS}''_{\gamma}$ in $\CE'_{\sigma}$ is disjoint from the image of any $\tilde{\CS}''_{\gamma'}$ in $\CE'_{\sigma}$ for all $\gamma' \neq \gamma$. 

For the upper bound, we remark that each $\gamma \in R_{\sigma}$ is either as above, or there exists a positive integer $\sigma' < \sigma$ and an element $\gamma' \in R_{\sigma'}$ such that the images of $\tilde{\CS}''_{\gamma'}$ in $\CE'_{\sigma'} \subset \CE'_{\sigma}$ and $\tilde{\CS}''_{\gamma}$ in $\CE'_{\sigma}$ intersect non-trivially. \end{proof}
Analogously to the compactification $\CE^{\dag}_{\sigma}$ of $\CE_{\sigma}$, we can construct a compactification $\CE'^{\dag}_{\sigma}$ of $\CE'_{\sigma}$. In analogy to Proposition \ref{upshot} we have the following proposition.
\begin{proposition} For any $\sigma' \leq \sigma$, the scheme ${\CE'}^{\dag}_{\sigma'}$ is a closed subspace of $\CE_{\sigma}^{\dag}$ and we have the equality
$$\bigcup_{\sigma' \leq \sigma} {\CE'}^{\dag}_{\sigma'} = \CE^{\dag}_{\sigma}.$$
Further, ${\CE'}^{\dag}_{\sigma}$ is the maximal closed subspace in $\CE^{\dag}_{\sigma}$ with the property that all of its irreducible components are of dimension $\sigma-1$. \end{proposition}
We leave the proof to the reader and obtain the following result.
\begin{proposition} There are inequalities
\begin{align*} \# \{\text{connected components of $\CE_{\sigma}'^{\dag}$}\} \leq \# \{\text{connected components of $\CE_{\sigma - 1}'^{\dag}$}\} \end{align*}
and
\begin{align*} \# \{\text{irreducible components of $\CE'_{\sigma}$}\} \leq \# \{\text{irreducible components of $\CE_{\sigma}'^{\dag}$}\}. \end{align*} \end{proposition}
\begin{proof} The proof of the first inequality goes analogously to the proof of the upper bound in the previous proposition. The second inequality is clear since $\CE'_{\sigma}$ is an open algebraic subspace in $\CE_{\sigma}'^{\dag}$. \end{proof}
\section{Torelli theorems for supersingular Enriques surfaces}\label{tor}
The schemes $\CE_{\sigma}$ are fine moduli spaces for $\Gamma'$-marked supersingular Enriques surfaces with Artin invariant at most $\sigma$, but their geometry is very complicated. However, it turns out that the much nicer schemes $\CQ_{\sigma}$ from Proposition \ref{per} are coarse moduli spaces for this moduli problem. The next proposition is a direct consequence of the Torelli theorem for supersingular K3 surfaces \cite{MR717616} and does not use any of our prior results.
\begin{proposition} Let $Y$ and $Y'$ be supersingular Enriques surfaces over an algebraically closed field $k$ of characteristic $p \geq 3$ which have universal K3 covers $X$ and $X'$ respectively. Let $\tilde{\phi} \colon \mathrm{NS}(Y) \rightarrow \mathrm{NS}(Y')$ be a morphism of lattices that maps the ample cone of $Y$ to the ample cone of $Y'$ and such that the induced morphism of lattices $\phi \colon \mathrm{NS}(X) \rightarrow \mathrm{NS}(X')$ extends via the first Chern map to an isomorphism $H^2_{\mathrm{crys}}(X/W) \rightarrow H^2_{\mathrm{crys}}(X'/W)$. Then $\tilde{\phi}$ is induced from an isomorphism $\tilde{\Phi} \colon Y \rightarrow Y'$ of supersingular Enriques surfaces.  \end{proposition}
\begin{proof} This follows immediately from a version of the Torelli theorem for supersingular K3 surfaces \cite[cf.\ Theorem II]{MR717616} and the fact that pullback along finite morphisms preserves ampleness of divisors. \end{proof}
We next want to show that the schemes $\CQ_{\sigma}$ are coarse moduli spaces for Enriques surfaces in the sense that their points parametrize isomorphism classes of Enriques surfaces without having to choose any kind of marking.
\begin{definition}  Recall from Theorem \ref{main} that there is a canonical étale surjective morphism $\pi^E_{\sigma} \colon \CE_{\sigma} \rightarrow \CQ_{\sigma}$. If $Y$ is a supersingular Enriques surface of Artin invariant $\sigma' \leq \sigma$ over an algebraically closed field $k$ of characteristic $p \geq 3$, we define the \emph{period $\pi^E_{\sigma}$ of $Y$ in $\CQ_{\sigma}$} to be $\pi^E_{\sigma}(Y)= \pi^E_{\sigma}(Y, \gamma)$, where $\gamma$ is any $\Gamma$-marking of $Y$. \end{definition}
The following proposition shows that $\pi^E_{\sigma}$ is well-defined and does not depend on the chosen marking. 
\begin{proposition}\label{nomark} Let $k$ be an algebraically closed field of characteristic $p \geq 3$, let $\sigma \leq 5$ be a positive integer and let $Y$ be a supersingular Enriques surface of Artin invariant at most $\sigma$ over $k$. For any choice of markings $\tilde{\gamma}_1 \colon \Gamma \rightarrow \mathrm{NS}(Y)$ and $\tilde{\gamma}_2 \colon \Gamma \rightarrow \mathrm{NS}(Y)$ we have an equality $\pi_{\sigma}^E(Y, \tilde{\gamma}_1) = \pi_{\sigma}^E(Y, \tilde{\gamma}_2)$. In other words, the period of $Y$ in $\CQ_{\sigma}$ is independent of the choice of a marking.  \end{proposition}

\begin{proof}
Let $Y$ be a supersingular Enriques surface which has the universal K3 covering $X \rightarrow Y$, let $\iota^{\ast} \colon \mathrm{NS}(Y)(2) \hookrightarrow \mathrm{NS}(X)$ be the associated embedding and let $\tilde{\gamma}_1 \colon \Gamma \stackrel{\simeq}\rightarrow \mathrm{NS}(Y)$ and $\tilde{\gamma}_2 \colon \Gamma \stackrel{\simeq}\rightarrow \mathrm{NS}(Y)$ be two choices of markings. Writing $\gamma_i = \iota^{\ast} \circ \tilde{\gamma}_i(2) \colon \Gamma(2) \hookrightarrow \mathrm{NS}(X)$, it follows from Proposition \ref{per} that the two points $\pi_{\sigma}^E(Y, \tilde{\gamma}_1)$ and $\pi_{\sigma}^E(Y, \tilde{\gamma}_2)$ are equal if and only if there exists some isometry $\varphi$ of $\mathrm{NS}(X)$ such that $\gamma_1 = \varphi \circ \gamma_2$. Let $K = \gamma_1(\Gamma(2))^{\perp} = \gamma_2(\Gamma(2))^{\perp}$ be the orthogonal complement of the image of $\Gamma(2)$ in $\mathrm{NS}(X)$. By \cite[Theorem 1.14.2]{1980IzMat..14..103N} we can choose for any automorphism $\psi \colon \gamma_1(\Gamma(2)) \rightarrow \gamma_1(\Gamma(2))$ an automorphism $\psi' \colon K \rightarrow K$ such that $(\psi, \psi')$ extends to an automorphism of $N_{\sigma}$. We take $\psi = \gamma_2 \circ \gamma_1$ and choose such a $\psi'$. Then $\varphi = (\psi, \psi')$ does the job.
\end{proof}
\begin{theorem}\label{torelli} Let $Y_1$ and $Y_2$ be supersingular Enriques surfaces. Then $Y_1$ and $Y_2$ are isomorphic if and only if  $\pi_{\sigma}^E(Y_1)=\pi_{\sigma}^E(Y_2)$ for some $\sigma \leq 5$. \end{theorem}
\begin{proof} It follows from Proposition \ref{nomark} that writing $\pi_{\sigma}^E(Y)$ makes sense since the period of $Y$ does not depend on the choice of a marking. We also directly obtain the 'only if' part of the theorem as a consequence of Proposition \ref{nomark}. We now let $Y_1$ and $Y_2$ be supersingular Enriques surfaces with the same period point and let $X_1 \rightarrow Y_1$ and $X_2 \rightarrow Y_2$ be their canonical K3 covers. We choose two markings $\tilde{\gamma}_1 \colon \Gamma \rightarrow \mathrm{NS}(Y_1)$ and $\tilde{\gamma}_2 \colon \Gamma \rightarrow \mathrm{NS}(Y_2)$. These induce $\Gamma(2)$-markings $\gamma_1 \colon \Gamma(2) \hookrightarrow \mathrm{NS}(X_1)$ and $\gamma_2 \colon \Gamma(2) \hookrightarrow \mathrm{NS}(X_2)$, and we may choose extensions of the morphisms $\gamma_i$ that are $N_{\sigma}$-markings $\eta_1 \colon N_{\sigma} \rightarrow \mathrm{NS}(X_1)$ and $\eta_2 \colon N_{\sigma} \rightarrow \mathrm{NS}(X_2)$, and after applying some isometry $\varphi \in O(N_{\sigma},\gamma)$ we may assume that the marked K3 surfaces $(X_1, \eta_1)$ and $(X_2, \eta_2)$ have the same period in $\CM_{\sigma}$. Hence, there exists an isomorphism of K3 crystals $\psi \colon H^2_{\mathrm{crys}}(X_1) \lra H^2_{\mathrm{crys}}(X_2)$ and a commutative diagram
\\
	\noindent\hspace*{\fill}
\begin{tikzpicture}
 \matrix (m) [matrix of math nodes, row sep=3em, column sep=3em]
    { \Gamma(2) & N_{\sigma}  & \mathrm{NS}(X_1)  & H^2_{\mathrm{crys}}(X_1) \\
      \Gamma(2) & N_{\sigma} & \mathrm{NS}(X_2) & H^2_{\mathrm{crys}}(X_2) \makebox[0pt][l]{.}  \\ };
	\path[-stealth]
 (m-1-1) edge node [above] {$\gamma$} (m-1-2) edge node [right] {$\operatorname{id}$} (m-2-1)
 (m-1-2) edge node [above] {$\eta_1$} (m-1-3) edge node [right] {$\operatorname{id}$} (m-2-2)
 (m-1-3) edge (m-1-4) edge node [right] {$\psi$} (m-2-3)
 (m-1-4) edge node [right] {$\psi$} (m-2-4)
 (m-2-1) edge node [above] {$\gamma$} (m-2-2)
 (m-2-2) edge node [above] {$\eta_2$} (m-2-3)
 (m-2-3) edge (m-2-4);
		\end{tikzpicture} 
		\hspace{\fill} \\
By a version of the Torelli theorem \cite[cf.\ Theorem II]{MR717616} the isomorphism $\psi$ is induced by some isomorphism of K3 surfaces $\Psi \colon X_1 \rightarrow X_2$. Since $\psi(\gamma_1(\Gamma(2))) = \gamma_2(\Gamma(2))$, if $\iota_1 \colon X_1 \rightarrow X_1$ and $\iota_2 \colon X_2 \rightarrow X_2$ are the involutions induced by the $\gamma_i$, we have that $\Psi \circ \iota_1 = \iota_2 \circ \Psi$ and it follows that the morphism $\Psi$ descends to an isomorphism of the Enriques quotients $\tilde{\Psi} \colon Y_1 \rightarrow Y_2$. 	 \end{proof}
\newpage
\section*{Notation and list of symbols}
For reference, here is a list of some of the objects we introduce. In general, if $\underline{\CF}$ is a representable functor, we write $\CF$ for the object representing this functor. 
\setlength{\glsdescwidth}{12.0cm}
\printglossary[type=symbols, style=super3col, title={}]

\printbibliography

\end{document}